\documentclass[11pt]{amsart}

\usepackage{amsmath, amssymb, amsthm, setspace, tikz-cd, enumitem}
\theoremstyle{plain}
\newtheorem{theorem}{Theorem}[section]
\newtheorem{lemma}[theorem]{Lemma}
\newtheorem{proposition}[theorem]{Proposition}
\newtheorem{corollary}[theorem]{Corollary}

\newtheorem*{theoremA}{Theorem A}
\newtheorem*{theoremB}{Theorem B}
\newtheorem*{theoremC}{Theorem C}
\newtheorem*{theoremD}{Theorem D}

\theoremstyle{definition}
\newtheorem{definition}[theorem]{Definition}
\newtheorem{exampletemp}[theorem]{Example}

\begin{document}

\title{Subalgebras of simple AF-algebras}
\author{Christopher Schafhauser}
\address{Department of Mathematics, University of Nebraska -- Lincoln, Lincoln, NE 68588}
\email{cschafhauser2@unl.ca}
\date{\today}
\begin{abstract}
  It is shown that if $A$ is a separable, exact C$^*$-algebra which satisfies the Universal Coefficient Theorem (UCT) and has a faithful, amenable trace, then $A$ admits a trace-preserving embedding into a simple, unital AF-algebra with a unique trace.  Modulo the UCT, this provides an abstract characterization of C$^*$-subalgebras of simple, unital AF-algebras.

  As a consequence, for a countable, discrete, amenable group $G$ acting on a second countable, locally compact, Hausdorff space $X$, $C_0(X) \rtimes_r G$ embeds into a simple, unital AF-algebra if, and only if, $X$ admits a faithful, invariant, Borel, probability measure.  Also, for any countable, discrete, amenable group $G$, the reduced group C$^*$-algebra $\mathrm{C}^*_r(G)$ admits a trace-preserving embedding into the universal UHF-algebra.
\end{abstract}

\maketitle

\section*{Introduction}

It follows from the work of Murray and von Neumann in \cite{MvN:RingsOfOperatorsIV} that there is a unique separably acting, hyperfinite $\mathrm{II}_1$-factor $\mathcal{R}$ and that any separably acting, hyperfinite, tracial von Neumann algebra admits a trace-preserving embedding into $\mathcal{R}$.  A fundamental result of Connes in \cite{Connes:InjectiveFactors} characterizes hyperfinite von Neumann algebras abstractly by showing the equivalence of hyperfiniteness and injectivity among von Neumann algebras.  Together, these results show that a separably acting, finite von Neumann algebra embeds into $\mathcal{R}$ if, and only if, it is hyperfinite, a result which has an instrumental role in the theory of subfactors initiated by Jones in \cite{Jones:Subfactors}.

A C$^*$-algebraic analogue of the Murray-von Neumann classification theorem is given by Elliott's celebrated result in \cite{Elliott:AF} classifying approximately finite-dimensional (AF) C$^*$-algebras and $^*$-homomorphisms between such algebras in terms of the non-stable $K_0$-group.  The problem of finding abstract characterizations of AF-algebras and C$^*$-subalgebras of AF-algebras was posed by Effros in \cite{Effros:Problems} with the latter problem motivated in part by the AF-embedding of the irrational rotation algebras of Pimsner and Voiculescu in \cite{PimsnerVoiculescu} which led to the classification of such C$^*$-algebras in terms of the angle of rotation.  Although an abstract characterization of AF-algebras among all C$^*$-algebras seems out of reach, such a characterization is possible among simple, unital C$^*$-algebras due to the remarkable success of Elliott's classification programme for simple, nuclear C$^*$-algebras over the last several decades with the final steps taken in \cite{GongLinNiu}, \cite{ElliottGongLinNiu}, and \cite{TWW:Annals}.

The problem of characterizing C$^*$-subalgebras of AF-algebras has received much attention over the last several decades with the standing conjecture being that a C$^*$-algebra embeds into an AF-algebra if, and only if, it is separable, exact, and quasidiagonal (see Section 7 of \cite{BlackadarKirchberg:MFAlgebras}).  Shortly after the Pimsner-Voiculescu AF-embedding result of \cite{PimsnerVoiculescu}, \mbox{Pimsner} characterized in \cite{Pimsner:AFE} which C$^*$-algebras of the form $C(X) \rtimes \mathbb{Z}$ can be embedded into an AF-algebra in terms of the underlying action of $\mathbb{Z}$ on $X$.  Similar results were obtained by Brown for crossed products of AF-algebras by $\mathbb{Z}$ in \cite{Brown:AFE} and for crossed products of UHF-algebras by $\mathbb{Z}^k$ in \cite{Brown:UHF}.  Many other partial results along these lines have appeared in \cite{Katsura:AFECuntzFlows}, \cite{Lin:AFEmbAHalg}, and \cite{Matui:AFE} for example.  The latter result in \cite{Matui:AFE} also has an important role in the work of Ozawa, R{\o}rdam, and Sato in \cite{OzawaRordamSato} showing that the C$^*$-algebra of an elementary amenable group embeds into an AF-algebra, a result which was later extended to all countable, discrete, amenable groups in \cite{TWW:Annals}.

Aside from crossed products, it is known that all residually stably finite, type I C$^*$-algebras embed into AF-algebras \cite{Spielberg:AFE}, separable, exact, residually finite-dimensional C$^*$-algebras satisfying the UCT embed into AF-algebras \cite{Dadarlat:TopKK}, and the cone over any separable, exact C$^*$-algebra embeds into an AF-algebra \cite{Ozawa:Cone}.  Also, combining Ozawa's result in \cite{Ozawa:Cone} with the techniques introduced by Spielberg in \cite{Spielberg:AFE}, Dadarlat has obtained AF-embeddings of continuous fields of C$^*$-algebras in \cite{Dadarlat:AFEmbContField} provided the base space is sufficiently connected and at least one fibre is AF-embeddable.  See Chapter 8 of \cite{BrownOzawa} for a well-written survey of the AF-embedding problem for C$^*$-algebras.

Despite the remarkable work on the AF-embedding problem given in the results above, in each of these results, the methods used are very specific to the class of C$^*$-algebras under consideration and shed very little light on the abstract AF-embedding problem.  This paper introduces a systematic method for producing embeddings into certain simple, unital AF-algebras.

It is well known that any C$^*$-subalgebra of a simple, unital AF-algebra must be separable and exact and must admit a faithful, amenable trace.  Modulo the Universal Coefficient Theorem (UCT) of \cite{RosenbergSchochet}, the present paper shows these are the only obstructions.  There are no known counterexamples to the UCT among separable, exact C$^*$-algebras with a faithful, amenable trace, and it is an important open problem whether all separable, nuclear C$^*$-algebras satisfy the UCT.

\begin{theoremA}
If $A$ is a separable, exact $C^*$-algebra which satisfies the UCT and admits a faithful, amenable trace, then there is a simple, unital AF-algebra $B$ with a unique trace and a trace-preserving embedding $A \hookrightarrow B$.
\end{theoremA}

Amenable traces were introduced by Connes in \cite{Connes:InjectiveFactors} in the von Neumann algebra setting and are characterized by the existence of almost multiplicative, almost trace-preserving, completely positive, contractive maps into $\mathbb{M}_n$, the algebra of $n \times n$ matrices over $\mathbb{C}$, where ``almost'' is measured in the 2-norm defined by the normalized trace on $\mathbb{M}_n$.  A key step in Connes's proof that injective $\mathrm{II}_1$-factors are hyperfinite consists of showing that the trace on an injective $\mathrm{II}_1$-factor is amenable.  In fact, the amenability of the trace characterizes injectivity of $\mathrm{II}_1$-factors.  Amenable traces were introduced in the C$^*$-algebraic setting by Kirchberg in \cite{Kirchberg:FactorizationProperty} and extensively developed by Brown in \cite{Brown:Memoirs}.

There is a very close connection between the theory of amenable traces on C$^*$-algebras and von Neumann algebras.  For example, if $A$ is an exact C$^*$-algebra, a trace $\tau_A$ on $A$ is amenable if, and only if, $\pi_{\tau_A}(A)''$ is hyperfinite.  Furthermore, if $\tau_A$ is faithful, then the GNS representation $\pi_{\tau_A} : A \rightarrow \pi_\tau(A)''$ is faithful, and hence any exact C$^*$-algebra $A$ with a faithful, amenable trace admits a trace-preserving embedding into a hyperfinite von Neumann algebra.  This observation, which can be viewed as a weak$^*$-version of Theorem A, is the starting point of the proof of the quasidiagonality theorem of Tikuisis, White, and Winter in \cite{TWW:Annals} and Gabe in \cite{Gabe:TWW} which, in turn, has an important role in the proof of Theorem A.

Building on the work of Ozawa, R{\o}rdam, and Sato in \cite{OzawaRordamSato}, it was shown in \cite{TWW:Annals} that if $G$ is a countable, discrete, amenable group, then C$^*_r(G)$ embeds into an AF-algebra, although these results give no control over the codomain AF-algebra.  Combining the techniques introduced in the present paper with the work of Higson and Kasparov \cite{HigsonKasparov:BaumConnes}, L\"uck \cite{Luck:TraceConjecture}, and Tu \cite{Tu:AmenableGroupoid} on the Baum-Connes conjecture yields a much sharper AF-embedding result showing group C$^*$-algebras embed into the universal UHF-algebra $\mathcal{Q} = \bigotimes_{n=1}^\infty \mathbb{M}_n$.

For a discrete group $G$, the group von Neumann algebra $L(G)$ and the reduced group C$^*$-algebra $\mathrm{C}^*_r(G)$ are equipped with the usual faithful trace given by $\sum_{g \in G} c_g \cdot g \mapsto c_e$ where $e \in G$ is the neutral element.

\begin{theoremB}
For a countable, discrete group $G$, the following are equivalent:
\begin{enumerate}
  \item $G$ is amenable;
  \item $L(G)$ admits a trace-preserving embedding into $\mathcal{R}$;
  \item $\mathrm{C}^*_r(G)$ admits a trace-preserving embedding into $\mathcal{Q}$.
\end{enumerate}
\end{theoremB}

In fact, the methods introduced here also yield an AF-embedding result for crossed products of abelian C$^*$-algebras by amenable groups in the spirit of Pimsner's result for crossed products $C(X) \rtimes \mathbb{Z}$ in \cite{Pimsner:AFE} and Lin's result for crossed products $C(X) \rtimes \mathbb{Z}^k$ in \cite{Lin:AFEmbAHalg}.

\begin{theoremC}
If $X$ is a second countable, locally compact, Hausdorff space and $G$ is a countable, discrete, amenable group acting on $X$, then $C_0(X) \rtimes G$ embeds into a simple, unital AF-algebra if, and only if, $X$ admits a faithful, $G$-invariant, Borel, probability measure.
\end{theoremC}

The new technical tool facilitating these results is a classification theorem for faithful $^*$-homomorphisms into certain AF-algebras.  The precise statement of Theorem D is given in Corollary \ref{cor:TheoremD}.  Theorems A, B, and C will be deduced from Theorem D at the beginning of Section \ref{sec:Applications}.

\begin{theoremD}
If $A$ is a separable, unital, exact $C^*$-algebra satisfying the UCT with a faithful, amenable trace and $B$ is a simple, unital AF-algebra with a unique trace and divisible $K_0$-group, then the unital, trace-preserving embeddings $A \rightarrow B$ are classified up to approximate unitary equivalence by their behaviour on the $K_0$-group.
\end{theoremD}

A more general classification result is given in Theorem \ref{thm:Classification} which does not require any inductive limit structure or nuclearity assumption on the codomain C$^*$-algebra.  Along with the applications to AF-embeddability listed above, this technical refinement of Theorem D also leads to a self-contained proof of an abstract characterization of AF-algebras among simple, unital C$^*$-algebras with a unique trace and divisible $K_0$-group in Corollary \ref{cor:AFAlg}.  These results also lead to new examples of MF algebras arising as reduced crossed products by free groups (see Corollary \ref{cor:KerrNowak}) building on \cite{KerrNowak, Rainone:MF, RainoneSchafhauser, Schafhauser:MF} and give AF-embedding results for $k$-graph algebras extending those of \cite{Schafhauser:Graph1} and \cite{ClarkHuefSims}.

Special cases of Theorem D have appeared in several places in the literature.  When $A$ is an AF-algebra or an A$\mathbb{T}$-algebra of real rank zero, Theorem D is a special case of classical results of Elliott in \cite{Elliott:AF} and \cite{Elliott:AT}, respectively.  For a commutative or, more generally, approximately homogeneous (AH) C$^*$-algebra $A$, classification results for embeddings from $A$ into an AF-algebra form a crucial part of the AF-embedding results for crossed products obtained by Lin in \cite{Lin:AFEmbAHalg}.  Also, the classification results for simple, nuclear C$^*$-algebras with tracial approximation structure, initiated by Lin in \cite{Lin:TAFClassification} and culminating in the remarkable work of Gong, Lin, and Niu in \cite{GongLinNiu}, depends heavily on classification results for embeddings between such algebras.

The power of Theorem D and what distinguishes Theorem D from existing classification results is the general hypothesis on the domain C$^*$-algebra.  The proof is motivated by the classical classification results for tracially AF-algebras in \cite{Dadarlat:TAFClassification} and \cite{Lin:TAF}, but it avoids making any use of the internal structure of the domain C$^*$-algebra and makes no use of inductive limit models such as the AH-algebras used in \cite{Lin:TAF}.  The proof of Theorem D, and the more general Theorem \ref{thm:Classification} below, serve as proof of concept for an abstract approach to Elliott's classification programme which does not depend on tracial approximations.  A more general result of this form will appear in forthcoming work of the author with Carri\'{o}n, Gabe, Tikuisis, and White.

\subsection*{On the proof of Theorem D}

The proof of Theorem D is heavily motivated by the classification theorem for separable, simple, unital, nuclear, tracially AF-algebras satisfying the UCT due to Lin in \cite{Lin:TAFClassification} and the classification of $^*$-homomorphisms between such algebras due to Dadarlat in \cite{Dadarlat:TAFClassification}, and, on a non-technical level, there are very close analogies between the proofs of these results.  The class of tracially AF-algebras was introduced by Lin in \cite{Lin:TAF} motivated in part by an approximation condition proved by Popa in \cite{Popa:Quasidiagonality} for simple, unital, quasidiagonal C$^*$-algebras of real rank zero.  Roughly, a tracially AF-algebra is a C$^*$-algebra $B$ which admits the following approximation condition in the spirit of Egoroff's Theorem: there is an approximately central projection in $B$ with large trace such that the corner $p B p$ is locally approximated by a finite-dimensional C$^*$-subalgebra where the approximations are in operator norm.

The basic strategy for the classification results in \cite{Dadarlat:TAFClassification} and \cite{Lin:TAFClassification} is to use the tracial finite-dimensional approximations for $A$ to produce approximately multiplicative maps from $A$ into a suitably well-behaved AH-algebra which approximately preserve the tracial data.  From here, one then perturbs these maps on tracially small corners of $A$ to adjust the behaviour of the maps on the infinitesimal elements of the $K$-theory groups.  Then, the structure of AH-algebras allows one to construct an embedding of the given AH-algebra into $B$ which implements an isomorphism on the invariant so that, upon composing this embedding with the approximately multiplicative maps into the AH-algebra, one obtains approximately multiplicative maps $A \rightarrow B$ with prescribed behaviour on $K$-theory and traces.  Together with a uniqueness result for approximately multiplicative maps from $A$ to $B$ relying on a Weyl-von Neumann type theorem for Hilbert module representations as in \cite{Lin:StableApproxEquivalence} or \cite{DadarlatEilers:Classification} and using the tracial approximation structure of $B$, an intertwining argument allows one to construct $^*$-homomorphisms from $A$ to $B$ with prescribed behaviour on $K$-theory and traces.

Changing perspective slightly, the almost multiplicative, almost trace-preserving maps $A \rightarrow B$ can be encoded as a trace-preserving $^*$-homomorphism $\psi : A \rightarrow B_\omega$ where $B_\omega$ denotes the norm ultrapower of $B$.  Adjusting the behaviour of these approximations on the tracially small corner of $A$ can be viewed as classifying, up to unitary equivalence, all $^*$-homomorphisms $\varphi : A \rightarrow B_\omega$ which are trace-zero perturbations of $\psi$.  It should be noted that this formalism is also explicitly used in the classical tracially AF-algebra classification results (usually with sequence algebras in place of ultrapowers, but this is mostly a matter of taste).

In the abstract setting of Theorem D, the existence of $\psi$ will follow directly from the quasidiagonality theorem of \cite{TWW:Annals} and \cite{Gabe:TWW} (see Theorem \ref{thm:TWW} below).  The new step is an abstract method for describing the trace-zero perturbations of $\psi$ up to unitary equivalence.  Consider the \emph{trace-kernel extension}
\[ \begin{tikzcd} 0 \arrow{r} & J_B \arrow{r}{j_B} & B_\omega \arrow{r}{q_B} & B^\omega \arrow{r} & 0 \end{tikzcd} \]
associated to $B$ where $B^\omega$ is the 2-norm ultrapower of $B$ and
\[ J_B = \{ b \in B_\omega : \tau_{B_\omega}(b^*b) = 0 \}. \]
Roughly, $B^\omega$ and $J_B$ will play the roles of the tracially large and small corners of $B$, respectively.

When $B$ has a unique trace and no non-zero, finite-dimensional representations, $B^\omega$ is a $\mathrm{II}_1$-factor, and hence all trace-preserving $^*$-homomorphisms $A \rightarrow B^\omega$ factor through the von Neumann algebra $\pi_{\tau_A}(A)''$ associated to $A$ and $\tau_A$.  When $A$ is exact and $\tau_A$ is amenable, $\pi_{\tau_A}(A)''$ is a hyperfinite von Neumann algebra, and hence the classification of normal $^*$-homomorphisms from hyperfinite von Neumann algebras to $\mathrm{II}_1$-factors yields a classification result for trace-preserving $^*$-homomorphisms $A \rightarrow B^\omega$ up to unitary equivalence.  It is through exploiting this observation that both the tracial approximation assumptions on $A$ and factoring through a model AH-algebra are avoided.

Having classification of trace-preserving $^*$-homomorphisms $A \rightarrow B^\omega \cong B_\omega / J_B$, the goal becomes to lift this classification along $q_B$ to obtain classification of $^*$-homomorphisms $A \rightarrow B_\omega$ up to unitary equivalence.  Using extension theoretic methods, it is shown in \cite{Schafhauser:Crelle} that (at least when $B \cong \mathcal{Q}$), modulo a certain $KK$-obstruction, a faithful, nuclear $^*$-homomorphism $A \rightarrow B^\omega$ can be lifted to a nuclear $^*$-homomorphism $A \rightarrow B_\omega$.  The strategy behind Theorem D is to use extension theoretic methods to show that the nuclear $^*$-homomorphisms $A \rightarrow B_\omega$ lifting a given faithful, nuclear $^*$-homomorphism $A \rightarrow B^\omega$ are parametrized up to unitary equivalence by the Kasparov group $KK_{\mathrm{nuc}}(A, J_B)$.\footnote{The C$^*$-algebra $J_B$ is not $\sigma$-unital.  In the actual proof, all computations will be done in a sufficiently large, separable C$^*$-subalgebra of $J_B$.}  Now, using the UCT and the $K$-theoretic assumptions on $B$, the group $KK_\mathrm{nuc}(A, J_B)$ is computed in terms of the $K_0$-groups of $A$ and $B_\omega$.  Together with an intertwining argument given in Section \ref{sec:Classification}, this will prove Theorem D.

In order to illustrate the lifting process suggested above, let us consider the uniqueness result (Proposition \ref{prop:UltraproductUniqueness}) in more detail; the corresponding existence result (Proposition \ref{prop:UltraproductExistence}) has a similar flavour.  Let $A$ and $B$ be as in Theorem D, and suppose $\varphi, \psi : A \rightarrow B_\omega$ are unital, full, nuclear $^*$-homomorphisms with $\tau_{B_\omega} \varphi = \tau_{B_\omega} \psi$ and $K_0(\varphi) = K_0(\psi)$.

Since $\varphi$ and $\psi$ agree on the trace on $B_\omega$, $q_B \varphi, q_B \psi : A \rightarrow B^\omega$ agree on the trace on $B^\omega$.  As $B^\omega$ is a $\mathrm{II}_1$-factor and $\varphi$ and $\psi$ are nuclear, $q_B \varphi$ and $q_B \psi$ are unitarily equivalent (see Proposition \ref{prop:FiniteFactorClassification} below).  Let $u \in B^\omega$ denote a unitary conjugating $q_B \psi$ to $q_B \varphi$.  If $F \in B_\omega$ with $q_B(F) = u$, then
\[ F^*F - 1_{B_\omega}, \hskip 3pt FF^* - 1_{B_\omega}, \hskip 3pt \varphi(a) - F \psi(a) F^* \in J_B \]
for all $a \in A$.\footnote{In fact, as the unitary group of $B^\omega$ is path-connected, $u$ lifts to a unitary.  For technical reasons, it will be helpful to take a unitary lift of $u$ and replace $\psi$ with a unitary conjugate of $\psi$ to arrange $F = 1_{B_\omega}$.}  Hence viewing $B_\omega$ as a C$^*$-subalgebra of the multiplier algebra $M(J_B)$ of $J_B$, the triple $(\varphi, \psi, F)$ defines an element in $KK_{\mathrm{nuc}}(A, J_B)$.

Under the hypotheses of Theorem D, since $K_0(\varphi) = K_0(\psi)$, this $KK$-class vanishes, so there are a $^*$-homomorphism $\pi : A \rightarrow M(J_B \otimes \mathcal{K})$ and a unitary $V \in M(J_B \otimes \mathcal{K})$ such that
\[ V - F \oplus 1_{M(J_B \otimes \mathcal{K})} \in J_B \otimes \mathcal{K} \quad \text{and} \quad \varphi \oplus \pi = \mathrm{ad}(V) (\psi \oplus \pi). \]
In Section \ref{sec:TraceKernel}, following the techniques introduced in \cite{Schafhauser:Crelle} for the case $B = \mathcal{Q}$, it is shown that, modulo separability issues, $J_B$ is stable and has the corona factorization property.  From here, the fullness of $\varphi$ and $\psi$ together with a Weyl-von Neumann type absorption theorem due to Elliott and Kucerovsky in \cite{ElliottKucerovsky} (Theorem \ref{thm:ElliottKucerovskyGabe} below) is used to remove the summand $\pi$ and show there is a unitary $U \in M(J_B)$ with
\[ U - F \in J_B \quad \text{and} \quad  \varphi = \mathrm{ad}(U) \psi. \]
But now, as $F \in B_\omega$ and $U - F \in J_B \subseteq B_\omega$, we have $U \in B_\omega$.  Since $U$ conjugates $\psi$ to $\varphi$ by construction, this shows the uniqueness result.

The paper is organized as follows.  In Section \ref{sec:Preliminaries}, some preliminary results on amenable traces are collected along with some general machinery for reducing problems to separable C$^*$-algebras.  Section \ref{sec:KKTheory} records some KK-theoretic prerequisites and obtains the non-stable KK-theoretic results needed in Section \ref{sec:UltraproductClassification}.  Section \ref{sec:TraceKernel} is devoted to the trace-kernel extension and proves certain extension-theoretic regularity conditions for the trace-kernel ideal.  The main classification results are given in Sections \ref{sec:UltraproductClassification} and \ref{sec:Classification}; the former section proves classification results for embeddings into ultrapower C$^*$-algebras, and the latter section restates these results in terms of approximate morphisms and, from here, obtains Theorem D via an intertwining argument.  Finally, Theorems A, B, and C along with some other consequences of Theorem D are proved in Section \ref{sec:Applications}.

\subsection*{Acknowledgements}

I would like to thank Jos\'{e} Carri\'{o}n, Jamie Gabe, Aaron Tikuisis, and Stuart White for several helpful conversations regarding this work and for their comments on early drafts of this paper.  I am also grateful to Matt Kennedy for pointing out the Choquet theoretic result used in the proof of Corollary \ref{cor:MFTrace}.  Finally, I thank the anonymous referees for several helpful comments on this paper.

\section{Preliminaries}\label{sec:Preliminaries}

\subsection{Amenable Traces}

Throughout, the word \emph{trace} is reserved for a tracial state on a C$^*$-algebra.  Given a C$^*$-algebra $A$ and a trace $\tau_A$ on $A$, let $\hat{\tau}_A : K_0(A) \rightarrow \mathbb{R}$ denote the induced state on $K_0(A)$.  In the case $A$ has a unit, $\hat{\tau}_A([p]) = (\tau \otimes \mathrm{Tr}_\mathcal{K})(p)$ for a projection $p \in A \otimes \mathcal{K}$ where $\mathrm{Tr}_\mathcal{K}$ is the usual tracial weight on the C$^*$-algebra $\mathcal{K}$ of compact operators on a separable, infinite-dimensional Hilbert space.

For $n \geq 1$, let $\tau_{\mathbb{M}_n}$ denote the unique trace on the C$^*$-algebra $\mathbb{M}_n$ of $n \times n$ matrices over $\mathbb{C}$ and define the \emph{2-norm} $\|\cdot\|_2$ on $\mathbb{M}_n$ by $\|a\|_2 = \tau_{\mathbb{M}_n}(a^*a)^{1/2}$ for all $a \in \mathbb{M}_n$.  A trace $\tau_A$ on a C$^*$-algebra $A$ is called \emph{amenable} if there is a net $\varphi_i : A \rightarrow \mathbb{M}_{n(i)}$ of completely positive, contractive maps with
\[ \| \varphi_i(aa') - \varphi_i(a) \varphi_i(a') \|_2 \rightarrow 0 \quad \text{and} \quad \tau_{\mathbb{M}_{n(i)}}(\varphi_i(a)) \rightarrow \tau_A(a) \]
for all $a, a' \in A$.  See \cite{Brown:Memoirs} or Chapter 6 of \cite{BrownOzawa} for a detailed treatment of amenable traces.  Note that by Theorem 4.2.1 of \cite{Brown:Memoirs}, all traces on nuclear C$^*$-algebras are amenable.

Exploiting the connections between amenable traces and hyperfinite von Neumann algebras given in Theorem 3.2.2 of \cite{Brown:Memoirs} leads to the following uniqueness result which is well known in the case when $A$ is nuclear.

\begin{proposition}\label{prop:FiniteFactorClassification}
If $A$ is a $C^*$-algebra, $\mathcal{M}$ is a finite factor, and $\varphi, \psi : A \rightarrow \mathcal{M}$ are weakly nuclear $^*$-homomorphisms such that $\tau_\mathcal{M} \varphi = \tau_\mathcal{M} \psi$, then there is a net of unitaries $(u_i)$ in $\mathcal{M}$ such that
\[ \| \varphi(a) - u_i \psi(a) u_i^* \|_2 \rightarrow 0 \]
for all $a \in A$.
\end{proposition}

\begin{proof}
Let $\tau_A = \tau_\mathcal{M} \varphi$ and note that $\varphi$ and $\psi$ induce normal $^*$-homomorphisms $\bar{\varphi}, \bar{\psi} : \pi_{\tau_A}(A)'' \rightarrow \mathcal{M}$ such that $\bar{\varphi}(\pi_{\tau_A}(a)) = \varphi(a)$ and $\bar{\psi}(\pi_{\tau_A}(a)) = \psi(a)$ for all $a \in A$.  As $\bar{\varphi}$ is faithful and normal, Lemma 1.5.11 of \cite{BrownOzawa} implies that there is a normal, completely positive map $\theta : \mathcal{M} \rightarrow \pi_{\tau_A}(A)''$ such that $\pi_\tau(a) = \theta(\varphi(a))$ for all $a \in A$.  Since $\varphi$ is weakly nuclear, $\pi_\tau(A)''$ is hyperfinite by the equivalence of (5) and (6) in Theorem 3.2.2 of \cite{Brown:Memoirs}.  The result follows from the classification of normal $^*$-homomorphisms from hyperfinite von Neumann algebras into finite factors; see the proof of Proposition 2.1 in \cite{CuipercaGiordanoNgNiu} for example.
\end{proof}

The following fundamental result was proved for nuclear C$^*$-algebras by Tikuisis, White, and Winter in \cite{TWW:Annals} and was extended to exact C$^*$-algebras by Gabe in \cite{Gabe:TWW} (see also \cite{Schafhauser:Crelle} for a short proof).  This result is the starting point for the existence result in Theorem D.

Let $\mathcal{Q} = \bigotimes_{n \geq 1} \mathbb{M}_n$ denote the universal UHF-algebra and let $\mathcal{Q}_\omega$ denote the norm ultrapower of $\mathcal{Q}$ with respect to a fixed free ultrafilter $\omega$ on the natural numbers.  Recall that a $^*$-homomorphism $\varphi : A \rightarrow B$ between C$^*$-algebras $A$ and $B$ is \emph{full} if for every non-zero $a \in A$, $\varphi(a)$ generates $B$ as an ideal.

\begin{theorem}\label{thm:TWW}
If $A$ is a separable, unital, exact $C^*$-algebra satisfying the UCT and $\tau_A$ is a faithful, amenable trace on $A$, then there is a unital, full, nuclear $^*$-homomorphism $\varphi : A \rightarrow \mathcal{Q}_\omega$ such that $\tau_{\mathcal{Q}_\omega} \varphi = \tau_A$.
\end{theorem}

The result is not quite stated this way in the references given above.  The existence of a unital, nuclear, trace-preserving $^*$-homomorphism $A \rightarrow \mathcal{Q}_\omega$ follows from Theorem 3.8 and Proposition 3.4(ii) in \cite{Gabe:TWW}, and this $^*$-homomorphism is necessarily full by Lemma 2.2 in \cite{TWW:Annals} and the faithfulness of $\tau_A$.

\subsection{Separability Issues}

Throughout the paper, several non-separable C$^*$-algebras such as ultraproducts and their trace-kernel ideals (as defined in Section \ref{sec:TraceKernel}) will be considered.  The lack of separability causes technical issues in certain arguments; this is especially the case with $KK$-theoretic considerations where all C$^*$-algebras are typically required to be separable or, at the very least, $\sigma$-unital.  This section collects some general methods for reducing problems to the separable setting.

\begin{definition}[Blackadar, Section II.8.5 of \cite{Blackadar:Encyclopedia}]
A property $(P)$ of C$^*$-algebras is called \emph{separably inheritable} if
\begin{enumerate}
  \item whenever $A$ is a C$^*$-algebra satisfying $(P)$ and $A_0$ is a separable C$^*$-subalgebra of $A$, there is a separable C$^*$-subalgebra $\hat{A}$ of $A$ which satisfies $(P)$ and contains $A_0$, and
  \item whenever $A_1 \hookrightarrow A_2 \hookrightarrow A_3 \hookrightarrow \cdots$ is an inductive system of separable C$^*$-algebras with injective connecting maps, if each $A_n$ satisfies $(P)$, then $\underset{\longrightarrow}{\lim} \, A_n$ satisfies $(P)$.
\end{enumerate}
\end{definition}

Many important properties of C$^*$-algebras are separably inheritable such as exactness, nuclearity, simplicity, real rank zero, and stable rank one, to name a few.  Also, the meet of countably many separably inheritable properties is separably inheritable.  See Section II.8.5 of \cite{Blackadar:Encyclopedia} for proofs of these facts and for many more examples of separably inheritable properties.

The following slight variation of separable inheritability will be useful.

\begin{definition}\label{defn:Separably}
Let $(P)$ be a property of separable C$^*$-algebras.  A C$^*$-algebra $A$ \emph{separably satisfies $(P)$} if whenever $A_0$ is a separable C$^*$-subalgebra of $A$, there is a separable C$^*$-subalgebra $\hat{A}$ of $A$ which satisfies $(P)$ and contains $A_0$.
\end{definition}

Note that if $(P)$ is a separably inheritable property of C$^*$-algebras and $A$ is a C$^*$-algebra satisfying $(P)$, then $A$ separably satisfies $(P)$.  Note also that if $(P)$ is a property of separable C$^*$-algebras preserved under sequential inductive limits with injective connecting maps, then separably $(P)$ is a separably inheritable property.

The following is analogous to II.8.5.3 in \cite{Blackadar:Encyclopedia}, and the same proof holds here.

\begin{proposition}
Let $(P_i)$ be a countable family of properties of separable $C^*$-algebras preserved under sequential inductive limits with injective connecting maps.  If $A$ is a $C^*$-algebra separably satisfying $(P_i)$ for each $i$, then $A$ separably satisfies the meet of the $(P_i)$.
\end{proposition}

The following result will be crucial for $KK$-theoretic considerations related to the trace-kernel extension defined in Section \ref{sec:TraceKernel}.

\begin{proposition}\label{prop:SeparableExtension}
Consider an extension
\[ \begin{tikzcd} 0 \arrow{r} & I \arrow{r}{j} & E \arrow{r}{q} & D \arrow{r} & 0 \end{tikzcd} \]
of $C^*$-algebras and suppose for each $X \in \{I, E, D \}$, $(P_X)$ is a property of separable $C^*$-algebras preserved under sequential inductive limits with injective connecting maps and $X$ separably satisfies $(P_X)$.  If for each $X \in \{ I, E, D \}$, a separable $C^*$-subalgebra $X_0$ of $X$ is given, then for each $X \in \{I, E, D\}$, there is a separable $C^*$-subalgebra $\hat{X}$ of $X$ which satisfies $(P_X)$ and contains $X_0$ and such that there is a homomorphism
\[ \begin{tikzcd}
   0 \arrow{r} & \hat{I} \arrow{r} \arrow{d} & \hat{E} \arrow{r} \arrow{d} & \hat{D} \arrow{r} \arrow{d} & 0 \\
   0 \arrow{r} & I \arrow{r}{j} & E \arrow{r}{q} & D \arrow{r} & 0
\end{tikzcd} \]
of extensions where the vertical arrows are the inclusion maps.
\end{proposition}

\begin{proof}
For $X \in \{I, E, D\}$, we construct an increasing sequence of separable C$^*$-subalgebras $(X_n)_{n=1}^\infty$ of $X$ containing $X_0$ such that for each $n \geq 1$, $X_n$ satisfies $(P_X)$,
\[ q(E_{n-1}) \subseteq D_n \subseteq q(E_n), \quad \text{and} \quad  I_{n-1} \subseteq j^{-1}(E_n) \subseteq I_n. \]
Assuming this has been done, for each $X \in \{ I, E, D \}$, let $\hat{X}$ be the closed union of the $X_n$ and note that $\hat{X}$ is a separable C$^*$-subalgebra of $X$.  As each $X_n$ satisfies $(P_X)$, so does $\hat{X}$ by hypothesis.  By construction, $q(\hat{E}) = \hat{D}$, and $j^{-1}(\hat{E}) = \hat{I}$, so the result follows.

We construct the desired C$^*$-subalgebras $I_n$, $E_n$, and $D_n$ inductively starting from the given C$^*$-subalgebras $I_0$, $D_0$, and $E_0$.  Assume $n \geq 1$ and $I_{n-1}$, $D_{n-1}$, and $E_{n-1}$ have been constructed.  The C$^*$-subalgebra of $D$ generated by $D_{n-1}$ and $q(E_{n-1})$ is separable, and hence there is a separable C$^*$-subalgebra $D_n$ of $D$ which satisfies $(P_D)$ and contains both $D_{n-1}$ and $q(E_{n-1})$.  Fix a countable, dense set $T_n \subseteq D_n$ and let $S_n \subseteq E$ be a countable set with $q(S_n) = T_n$.  Then the C$^*$-subalgebra of $E$ generated by $j(I_{n-1})$, $E_{n-1}$, and $S_n$ is separable, and hence there is a separable C$^*$-subalgebra $E_n$ of $E$ which satisfies $(P_E)$ and contains $j(I_{n-1})$, $E_{n-1}$, and $S_n$.  Then $T_n = q(S_n) \subseteq q(E_n)$ since $S_n \subseteq E_n$, and as $T_n$ is dense in $D_n$ and the $^*$-homomorphism $q|_{E_n}$ has closed range, $D_n \subseteq q(E_n)$.  Also, as $j(I_{n-1}) \subseteq E_n$ and $j$ is injective, $I_{n-1} \subseteq j^{-1}(E_n)$.  Finally, as $j^{-1}(E_n)$ is a separable C$^*$-subalgebra of $I$, there is a separable C$^*$-subalgebra $I_n$ of $I$ which satisfies $(P_I)$ and contains $j^{-1}(E_n)$.  This completes the construction.
\end{proof}

\begin{corollary}\label{cor:SeparableExtensionPermenence}
Let $(P)$ be a property of separable $C^*$-algebras preserved under sequential inductive limits with injective connecting maps.  If $(P)$ is preserved by ideals, quotients, or extensions of separable $C^*$-algebras, then separably $(P)$ has the same permanence property among all $C^*$-algebras.
\end{corollary}

\begin{proof}
We only consider the case of extensions as the other two results are similar.  Let $(P)$ be a property of separable C$^*$-algebras preserved by extensions.  Suppose $I$ is an ideal of a C$^*$-algebra $A$ such that $I$ and $A / I$ both separably satisfy $(P)$ and suppose $A_0$ is a separable C$^*$-subalgebra of $A$.  By Proposition \ref{prop:SeparableExtension}, there are separable C$^*$-subalgebras $\hat{I}$, $\hat{A}$, and $\hat{B}$ of $I$, $A$, and $A / I$, respectively, such that $\hat{I}$ and $\hat{B}$ satisfy $(P)$, $\hat{A}$ contains $A_0$, and there is a homomorphism
\[ \begin{tikzcd}
   0 \arrow{r} & \hat{I} \arrow{r} \arrow{d} & \hat{A} \arrow{r} \arrow{d} & \hat{B} \arrow{r} \arrow{d} & 0 \\
   0 \arrow{r} & I \arrow{r} & A \arrow{r} & A / I \arrow{r} & 0
\end{tikzcd} \]
of extensions where the vertical arrows are the inclusion maps.  Now $\hat{A}$ is a separable C$^*$-subalgebra of $A$ satisfying $(P)$ and containing $A_0$.
\end{proof}

The method for reducing to separable C$^*$-algebras given here also behaves well with hereditary subalgebras.

\begin{proposition}\label{prop:SeparableHereditaryPermenance}
If $(P)$ is a property of separable $C^*$-algebras preserved by hereditary subalgebras, then the property separably $(P)$ is preserved by hereditary subalgebras.
\end{proposition}

\begin{proof}
Suppose $A$ is a C$^*$-algebra separably satisfying $(P)$ and $B \subseteq A$ is a hereditary subalgebra.  Let $B_0$ be a separable C$^*$-subalgebra of $B$.  There is a separable C$^*$-subalgebra $\hat{A}$ of $A$ which satisfies $(P)$ and contains $B_0$.  Now, $\hat{B} := B_0 \hat{A} B_0$ is a separable C$^*$-subalgebra of $B$ containing $B_0$.  Moreover, $\hat{B}$ is a hereditary subalgebra of $\hat{A}$ whence satisfies $(P)$.
\end{proof}

The following result will be used heavily in Section \ref{sec:UltraproductClassification}.

\begin{proposition}\label{prop:SeparableCorestriction}
Suppose $A$ and $B$ are $C^*$-algebras such that $A$ is separable and $B$ is unital.  If $\varphi : A \rightarrow B$ is a full, nuclear $^*$-homomorphism, there is a separable, unital $C^*$-subalgebra $B_0$ of $B$ such that $\varphi(A) \subseteq B_0$ and the corestriction of $\varphi$ to $B_0$ is full and nuclear.
\end{proposition}

\begin{proof}
As $A$ is separable and $\varphi$ is nuclear, for each integer $n \geq 1$, there are an integer $d(n) \geq 1$ and completely positive maps $\theta_n : A \rightarrow \mathbb{M}_{d(n)}$ and $\rho_n : \mathbb{M}_{d(n)} \rightarrow B$ such that
\[ \| \rho_n(\theta_n(a)) - \varphi(a) \| \rightarrow 0 \]
for all $a \in A$.

By Proposition II.8.5.7 of \cite{Blackadar:Encyclopedia}, there is a sequence $(a_n)_{n=1}^\infty \subseteq A \setminus \{0\}$ such that for every non-zero ideal $I \subseteq A$, there is an $n \geq 1$ with $a_n \in I$.  As $\varphi$ is full and $B$ is unital, for each $n \geq 1$, there are an integer $k(n) \geq 1$ and elements $b_{n, i}, b'_{n, i} \in B$ for $i = 1, \ldots, k(n)$ such that $\sum_{i=1}^{k(n)} b_{n, i} \varphi(a_n) b'_{n, i} = 1_B$.

Let $B_0$ denote the C$^*$-subalgebra of $B$ generated by $\varphi(A)$, $\rho_n(\mathbb{M}_{d(n)})$, $b_{n, i}$, and $b'_{n, i}$ for $i = 1, \ldots, k(n)$ and $n \geq 1$.  Then $B_0$ is separable, and if $\varphi_0$ is the corestriction of $\varphi$ to $B_0$, then $\varphi_0$ is nuclear.  Suppose $a \in A \setminus \{0\}$ and let $I$ denote the ideal of $B_0$ generated by $\varphi(a)$.  Then $\varphi^{-1}(I)$ is an ideal in $A$ which is non-zero as $a \in \varphi^{-1}(I)$, and hence there is an integer $n \geq 1$ such that $a_n \in \varphi^{-1}(I)$.  Now, $\varphi(a_n) \in I$, and since $\varphi(a_n)$ is full in $B_0$ by construction, $I = B_0$.  This shows that $\varphi_0$ is full.
\end{proof}

A version of the following result appeared in an early version of \cite{Schafhauser:Crelle}.

\begin{proposition}\label{prop:NonseparableHomGroups}
Suppose $G$ is a countable, abelian group and $A$ is a $C^*$-algebra.  For $i = 0, 1$, the natural group homomorphisms
\begin{align*}
\underset{\longrightarrow}{\lim} \, \mathrm{Hom}_\mathbb{Z}(G, K_i(A_0)) &\longrightarrow \mathrm{Hom}_\mathbb{Z}(G, K_i(A))
\intertext{and}
\underset{\longrightarrow}{\lim} \, \mathrm{Ext}_\mathbb{Z}^1(G, K_i(A_0)) &\longrightarrow \mathrm{Ext}_\mathbb{Z}^1(G, K_i(A))
\end{align*}
are isomorphisms where the limit is taken over all separable $C^*$-subalgebras $A_0$ of $A$.
\end{proposition}

\begin{proof}
As $G$ is a countable, abelian group, there is an extension
\[ \begin{tikzcd} 0 \arrow{r} & \mathbb{Z}X \arrow{r} & \mathbb{Z}Y \arrow{r} & G \arrow{r} & 0 \end{tikzcd} \]
for countable sets $X$ and $Y$ where $\mathbb{Z}X$ and $\mathbb{Z}Y$ denote the free abelian groups generated by $X$ and $Y$, respectively.  For every C$^*$-algebra $B$ and for $i = 0, 1$, there is a natural exact sequence
\[ \begin{tikzcd}[column sep = small] \mathrm{Hom}_\mathbb{Z}(G, K_i(B)) \arrow[tail]{r} & \mathrm{Hom}_\mathbb{Z}(\mathbb{Z}Y, K_i(B)) \arrow{r} & \mathrm{Hom}_\mathbb{Z}(\mathbb{Z}X, K_i(B)) \arrow[two heads]{r} & \mathrm{Ext}_\mathbb{Z}^1(G, K_i(B)). \end{tikzcd} \]
As inductive limits preserve exact sequences, it is enough to prove that the natural map
\[ \underset{\longrightarrow}{\lim} \, \mathrm{Hom}_\mathbb{Z}(\mathbb{Z}X, K_i(A_0)) \longrightarrow \mathrm{Hom}_\mathbb{Z}(\mathbb{Z}X, K_i(A)) \]
is an isomorphism for each countable set $X$.

Adding a unit to $A$ if necessary, we may assume $A$ is unital.  We only prove the case when $i = 0$ as the case $i = 1$ then follows by Bott periodicity (or by a similar argument).

To show surjectivity, let $f : \mathbb{Z} X \rightarrow K_0(A)$ be given.  For each $x \in X$, there are an integer $n(x) \geq 1$ and projections $p_x$ and $q_x$ in $\mathbb{M}_{n(x)}(A)$ with $f(x) = [p_x] - [q_x]$.  If $A_0$ denotes the unital C$^*$-subalgebra of $A$ generated by the entries of the projections $p_x$ and $q_x$, then there is a group homomorphism $f_0 : \mathbb{Z}X \rightarrow K_0(A_0)$ given by $f_0(x) = [p_x] - [q_x]$.  If $\iota_0 : A_0 \rightarrow A$ denotes the inclusion, then $K_0(\iota_0) f_0 = f$.

To show injectivity, let $A_0 \subseteq A$ be a separable, unital C$^*$-subalgebra, let $\iota_0 : A_0 \rightarrow A$ denote the inclusion, and suppose $f, g : \mathbb{Z}X \rightarrow K_0(A_0)$ are such that $K_0(\iota_0) f = K_0(\iota_0) g$.  For each $x \in X$, there are an integer $n(x) \geq 1$ and projections $p_x, q_x, p'_x, q'_x \in \mathbb{M}_{n(x)}(A_0)$ such that $f(x) = [p_x] - [q_x]$ and $g(x) = [p'_x] - [q'_x]$.  For each $x \in X$, there are an integer $k(x) \geq 1$ and a partial isometry $v_x \in \mathbb{M}_{2n(x) + k(x)}(A)$ with
\[ v_x^* v_x = p_x \oplus q'_x \oplus 1_A^{\oplus k(x)} \quad \text{and} \quad v_x v_x^* = p'_x \oplus q_x \oplus 1_A^{\oplus k(x)} \]
Let $A_1$ denote the C$^*$-subalgebra of $A$ generated by $A_0$ and the entries of each $v_x$ for $x \in X$.  If $\iota_{1, 0} : A_0 \hookrightarrow A_1$ denotes the inclusion, then $K_0(\iota_{1, 0}) f = K_0(\iota_{1, 0}) g$.
\end{proof}

\subsection{Tensorial Absorption and Separability}

Recall from \cite{TomsWinter:SSA} that a separable, unital, infinite-dimensional C$^*$-algebra $D$ is \emph{strongly self-absorbing} if there is an isomorphism $D \rightarrow D \otimes D$ approximately unitarily equivalent to the first factor embedding.  For a strongly self-absorbing C$^*$-algebra $D$, a separable C$^*$-algebra $A$ is \emph{$D$-stable} if $A \otimes D \cong A$.  The only strongly self-absorbing C$^*$-algebra needed here is the universal UHF-algebra $\mathcal{Q}$.

A local condition characterizing $D$-stability for separable, unital C$^*$-algebras can be extracted from Theorem 2.2 of \cite{TomsWinter:SSA} which shows a separable, unital C$^*$-algebra $A$ is $D$-stable if, and only if, there is a unital embedding of $D$ into the \emph{central sequence algebra} $A_\infty \cap A'$ of $A$ where $A_\infty := \ell^\infty(A) / c_0(A)$.  This local characterization extends to the non-separable setting with $D$-stability replaced by separable $D$-stability.

\begin{lemma}\label{lemma:LocalDStability}
For a strongly self-absorbing $C^*$-algebra $D$, a unital $C^*$-algebra $A$ is separably $D$-stable if, and only if, for every finite set $\mathcal{F} \subseteq A$, for every finite set $\mathcal{G} \subseteq D$, and for every $\varepsilon > 0$, there is a unital, completely positive map $\varphi : D \rightarrow A$ such that
\[ \|\varphi(d d') - \varphi(d) \varphi(d') \| < \varepsilon \quad \text{and} \quad \| a \varphi(d) - \varphi(d) a \| < \varepsilon \]
for all $a \in \mathcal{F}$ and $d, d' \in \mathcal{G}$.
\end{lemma}

\begin{proof}
Assume first that $A$ is separably $D$-stable and fix a finite set $\mathcal{F} \subseteq A$, a finite set $\mathcal{G} \subseteq \mathcal{D}$, and $\varepsilon > 0$. There is a separable, unital, $D$-stable C$^*$-subalgebra $\hat{A}$ of $A$ containing $\mathcal{F}$.  By Theorem 2.2 of \cite{TomsWinter:SSA}, there is a unital embedding $\varphi_\infty : D \rightarrow \hat{A}_\infty \cap \hat{A}'$.  As $D$ is nuclear, the Choi-Effros lifting theorem implies that there are unital, completely positive maps $\varphi_n : D \rightarrow \hat{A}$ representing $\varphi_\infty$.  Hence
\[ \lim_{n \rightarrow \infty} \| \varphi_n(d d') - \varphi_n(d) \varphi_n(d') \| = \lim_{n \rightarrow \infty} \| a \varphi_n(d) - \varphi_n(d) a \| = 0 \]
for all $a \in \hat{A}$ and $d, d' \in D$.  Take $\varphi = \varphi_n$ for some sufficiently large $n$.

Conversely, suppose the approximation condition holds and let $A_0$ be a separable, unital C$^*$-subalgebra of $A$.  Let $\mathcal{F}_{0, n}$ be an increasing sequence of finite subsets of $A_0$ with dense union and let $\mathcal{G}_n$ be an increasing sequence of finite subsets of $D$ with dense union.  There are unital, completely positive maps $\varphi_{0, n} : D \rightarrow A$ such that
\[ \| \varphi_{0, n}(d d') - \varphi_{0, n}(d) \varphi_{0, n}(d') \| < \frac1n \quad \text{and} \quad \| a \varphi_{0, n}(d) - \varphi_{0, n}(d) a \| < \frac1n  \]
for all $a \in \mathcal{F}_{0, n}$ and $d, d' \in \mathcal{G}_n$.  Let $A_1$ denote the C$^*$-subalgebra of $A$ generated by $A_0$ and $\varphi_{0, n}(D)$ for each $n \geq 1$ and note that $A_1$ is separable.

Iterating this argument, there are an increasing sequence of separable C$^*$-subalgebras $A_k$ of $A$ and sequences of unital, completely positive maps $\varphi_{k, n} : D \rightarrow A_{k+1}$ with
\[ \lim_{n \rightarrow \infty} \| \varphi_{k,n}(d d') - \varphi_{k,n}(d) \varphi_{k, n}(d') \| = \lim_{n \rightarrow \infty} \| a \varphi_{k,n}(d) - \varphi_{k,n}(d) a \| = 0 \]
for all $a \in A_k$, $d, d' \in D$, and $k \geq 0$.  Let $\hat{A} \subseteq A$ denote the closed union of the $A_k$.  A reindexing argument produces a sequence of unital, completely positive maps $\psi_n : D \rightarrow \hat{A}$ such that
\[ \lim_{n \rightarrow \infty} \| \psi_n(d d') - \psi_n(d) \psi_n(d') \| = \lim_{n \rightarrow \infty} \| a \psi_n(d) - \psi_n(d) a \| = 0 \]
for all $a \in \hat{A}$ and $d, d' \in D$.  The sequence $\psi_n$ induces a unital $^*$-homomorphism $D \rightarrow \hat{A}_\infty \cap \hat{A}'$, and since $\hat{A}$ is separable, $\hat{A}$ is $D$-stable by Theorem 2.2 of \cite{TomsWinter:SSA}.  As $\hat{A}$ contains $A_0$ by construction, this shows $A$ is separably $D$-stable.
\end{proof}

The next proposition collects some permanence properties of separable $D$-stability.

\begin{proposition}\label{prop:DStablePermanence}
For a strongly self-absorbing $C^*$-algebra $D$, hereditary subalgebras, quotients, and extensions of separably $D$-stable $C^*$-algebras are separably $D$-stable, and $\ell^\infty$-products and ultraproducts of unital, separably $D$-stable $C^*$-algebras are separably $D$-stable.
\end{proposition}

\begin{proof}
By Corollary 3.4 in \cite{TomsWinter:SSA}, $D$-stability is preserved by sequential inductive limits of separable C$^*$-algebras and by Corollaries 3.1 and 3.3 and Theorem 4.3 of \cite{TomsWinter:SSA}, hereditary subalgebras, quotients, and extensions of separable, $D$-stable C$^*$-algebras are $D$-stable.  Now, separable $D$-stability passes to hereditary subalgebras by Proposition \ref{prop:SeparableHereditaryPermenance} and is preserved by quotients and extensions by Corollary \ref{cor:SeparableExtensionPermenence}.  For $\ell^\infty$-products, the result follows from Lemma \ref{lemma:LocalDStability} by choosing approximately central approximate morphisms into each factor of the product and taking the product of these maps.  For ultraproducts, the result follows from the result for $\ell^\infty$-products and quotients.
\end{proof}

\section{Some $KK$-theory}\label{sec:KKTheory}

This section contains a brief overview of $KK$-theory and collects the results on absorbing representations which will be needed in the classification results in Section \ref{sec:UltraproductClassification}.  With one exception in the proof of Proposition \ref{prop:AdmissibleKernel}, we will work exclusively with the nuclear $KK$-bifunctor $KK_\mathrm{nuc}(-, -)$ introduced by Skandalis in \cite{Skandalis:KKnuc}.

\subsection{Basics of $KK$-theory}

Let $A$ be a separable C$^*$-algebra and let $B$ be a $\sigma$-unital C$^*$-algebra.  The word \emph{representation} will refer to a $^*$-homomorphism $A \rightarrow M(B \otimes \mathcal{K})$.  A representation $\varphi : A \rightarrow M(B \otimes \mathcal{K})$ is called \emph{weakly nuclear} if the completely positive map $A \rightarrow B \otimes \mathcal{K}$ given by $a \mapsto b^* \varphi(a) b$ is nuclear for all $b \in B \otimes \mathcal{K}$.

Let $\mathcal{E}_\mathrm{nuc}(A, B)$ denote the set of pairs $(\varphi, \psi)$ such that $\varphi, \psi : A \rightarrow M(B \otimes \mathcal{K})$ are weakly nuclear representations with $\varphi(a) - \psi(a) \in B \otimes \mathcal{K}$ for all $a \in A$.  Such a pair $(\varphi, \psi)$ is called a (\emph{weakly nuclear}) \emph{Cuntz pair}.  A \emph{homotopy} between Cuntz pairs $(\varphi_0, \psi_0)$ and $(\varphi_1, \psi_1)$ in $\mathcal{E}_\mathrm{nuc}(A, B)$ is a Cuntz pair $(\Phi, \Psi) \in \mathcal{E}_\mathrm{nuc}(A, C([0, 1], B))$ such that for $t \in \{0, 1\}$, composing $\Phi$ and $\Psi$ with the evaluation map $M(C([0, 1], B) \otimes \mathcal{K}) \twoheadrightarrow M(B \otimes \mathcal{K})$ at $t$ produces $\varphi_t$ and $\psi_t$, respectively.  Let $KK_\mathrm{nuc}(A, B)$ denote the set of homotopy classes of Cuntz pairs in $\mathcal{E}_\mathrm{nuc}(A, B)$ and let $[\varphi, \psi]$ denote the class of a Cuntz pair $(\varphi, \psi)$ in $KK_\mathrm{nuc}(A, B)$.

Let $s_1, s_2 \in M(B \otimes \mathcal{K})$ be isometries with $s_1 s_1^* + s_2 s_2^* = 1$.  Given representations $\theta, \rho : A \rightarrow M(B \otimes \mathcal{K})$, define a representation $\theta \oplus_{s_1, s_2} \rho = s_1 \theta(\cdot) s_1^* + s_2 \rho(\cdot) s_2^*$ called the \emph{Cuntz sum} with respect to $s_1$ and $s_2$.  For another choice of isometries $t_1, t_2 \in M(B \otimes \mathcal{K})$ with $t_1 t_1^* + t_2 t_2^* = 1$, $u = t_1 s_1^* + t_2 s_2^*$ is a unitary with $\mathrm{ad}(u)(\theta \oplus_{s_1, s_2} \rho) = \theta \oplus_{t_1, t_2} \rho$, and hence, up to unitary equivalence, the Cuntz sum is independent of the choice of $s_1$ and $s_2$.

If $(\varphi_1, \psi_1)$ and $(\varphi_2, \psi_2)$ are Cuntz pairs in $\mathcal{E}_\mathrm{nuc}(A, B)$, then $(\varphi_1 \oplus_{s_1, s_2} \varphi_2, \psi_1 \oplus_{s_1, s_2} \psi_2)$ is also a Cuntz pair in $\mathcal{E}_\mathrm{nuc}(A, B)$.    Since the unitary group of $M(B \otimes \mathcal{K})$ is path-connected in the operator norm topology by the main result of \cite{CuntzHigson:KupiersTheorem} (see also \cite{Mingo:KupiersTheorem} for the case when $B$ is unital), the class $[\varphi_1 \oplus_{s_1, s_2} \varphi_2, \psi_1 \oplus_{s_1, s_2} \psi_2]$ in $KK_\mathrm{nuc}(A, B)$ is independent of the choice of isometries $s_1$ and $s_2$; abusing notation, this element will be written as $[\varphi_1 \oplus \varphi_2, \psi_1 \oplus \psi_2]$.  The set $KK_\mathrm{nuc}(A, B)$ is an abelian group with addition given by Cuntz sum.

If $\varphi : A \rightarrow B$ is a nuclear $^*$-homomorphism and $p \in \mathcal{K}$ is a rank one projection, define a representation $\varphi_p : A \rightarrow M(B \otimes \mathcal{K})$ by $\varphi_p(a) = \varphi(a) \otimes p$ for $a \in A$.  Then $(\varphi_p, 0)$ defines a Cuntz pair in $\mathcal{E}_\mathrm{nuc}(A, B)$ and the corresponding element of $KK_\mathrm{nuc}(A, B)$ is denoted by $[\varphi]$.  The element $[\varphi]$ is independent of the choice of the rank one projection $p$.

Given a separable C$^*$-algebra $A$ and a $^*$-homomorphism $\theta : B \rightarrow D$ between $\sigma$-unital C$^*$-algebras $B$ and $D$, there is an induced group homomorphism
\[ \theta_* : KK_\mathrm{nuc}(A, B) \rightarrow KK_\mathrm{nuc}(A, D). \]
In this way, $KK_\mathrm{nuc}(A, -)$ becomes a covariant functor from the category of $\sigma$-unital C$^*$-algebras to the category of abelian groups; see \cite{Skandalis:KKnuc} for the details.  We will only need an explicit computation of $\theta_*$ in the following special case.

\begin{proposition}\label{prop:KasparovProduct}
Suppose $A$ and $E$ are separable $C^*$-algebras, $I \subseteq E$ is an ideal in $E$ with $I \otimes \mathcal{K} \cong I$, and $\varphi, \psi : A \rightarrow E$ are nuclear $^*$-homomorphisms with $\varphi(a) - \psi(a) \in I$ for all $a \in A$.  Let $\lambda : E \rightarrow M(I)$ denote the canonical $^*$-homomorphism and note that $(\lambda \varphi, \lambda \psi) \in \mathcal{E}_\mathrm{nuc}(A, I)$.

If $j : I \rightarrow E$ denotes the inclusion map, then $j_*[\lambda \varphi, \lambda \psi] = [\varphi] - [\psi]$ in $KK_\mathrm{nuc}(A, E)$.
\end{proposition}

\begin{proof}
As in \cite{Skandalis:KKnuc} (see also Chapter 17 of \cite{Blackadar:KTheory}), $KK_\mathrm{nuc}(A, I)$ can be realized as homotopy classes of Kasparov modules $(\theta_0, \theta_1, F)$ where $\theta_i : A \rightarrow \mathcal{B}(K_i)$ is a weakly nuclear representation of $A$ on a countably generated Hilbert $I$-module $K_i$ and $F : K_0 \rightarrow K_1$ is an adjointable operator which, modulo the compacts, is a unitary intertwining of $\theta_0$ and $\theta_1$.

Viewing $I$ has a Hilbert module over itself, the Cuntz pair $(\lambda \varphi, \lambda \psi)$ defines the Kasparov module $(\lambda \varphi, \lambda \psi, 1_{M(I)})$ which defines the element $[\lambda \varphi, \lambda \psi] \in KK_\mathrm{nuc}(A, I)$.  If $H_0 = I$ viewed as a Hilbert $E$-module, then there is a natural isomorphism $M(I) \rightarrow \mathcal{B}(H_0)$.  Let $\varphi_0, \psi_0 : A \rightarrow \mathcal{B}(H_0)$ denote the representations given by composing $\lambda \varphi$ and $\lambda \psi$ with this isomorphism.  Then $j_*[\lambda \varphi, \lambda \psi, 1_{M(I)}] = [\varphi_0, \psi_0, 1_{\mathcal{B}(H_0)}]$ in $KK_\mathrm{nuc}(A, E)$ by the construction of $j_*$ given in \cite{Skandalis:KKnuc} (see also Section 17.8 of \cite{Blackadar:KTheory}).

Consider $H_1 = E$ as a Hilbert $E$-module.  If $\varphi_1, \psi_1 : A \rightarrow \mathcal{B}(H_1)$ denote the representations induced by $\varphi$ and $\psi$, respectively, then $[\varphi_1, \psi_1, 1_{\mathcal{B}(H_1)}] = [\varphi] - [\psi]$ in $KK_\mathrm{nuc}(A, I)$.  It suffices to show $(\varphi_0, \psi_0, 1_{\mathcal{B}(H_0)})$ and $(\varphi_1, \psi_1, 1_{\mathcal{B}(H_1)})$ are homotopic.  To this end, define
\[ H = \{ f \in C([0, 1], E) : f(0) \in I \} \]
and view $H$ as a Hilbert $C([0,1], E)$-module in the natural way.  Define $\Phi, \Psi : A \rightarrow \mathcal{B}(H)$ by
\[ \Phi(a)(f)(t) = \varphi(a)f(t) \quad \text{and} \quad \Psi(a)(f)(t) = \psi(a)f(t) \]
for all $a \in A$, $f \in C([0, 1], E)$, and $t \in [0, 1]$.  Then $(\Phi, \Psi, 1_{\mathcal{B}(H)})$ is a homotopy from $(\varphi_0, \psi_0, 1_{\mathcal{B}(H_0)})$ to $(\varphi_1, \psi_1, 1_{\mathcal{B}(H_1)})$.
\end{proof}

For a separable C$^*$-algebras $A$, a $\sigma$-unital C$^*$-algebras, $B$, and a representation $\psi : A \rightarrow M(B \otimes \mathcal{K})$, define the \emph{infinite repeat} $\psi_\infty : A \rightarrow M(B \otimes \mathcal{K})$ as follows.  Let $(s_n)_{n=1}^\infty$ be a sequence of isometries in $M(B \otimes \mathcal{K})$ such that $\sum_{n=1}^\infty s_n s_n^* = 1$ and let
\[ \psi_\infty(a) = \sum_{n=1}^\infty s_n \psi(a) s_n^* \in M(B \otimes \mathcal{K}) \]
for $a \in A$ where convergence is in the strict topology.  Then $\psi_\infty$ is a $^*$-homomorphisms, and up to unitary equivalence, $\psi_\infty$ is independent of the choice of the sequence $(s_n)_{n=1}^\infty$; in particular, $\psi \oplus \psi_\infty$ and $\psi$ are unitarily equivalent.  Note also that if $\psi$ is weakly nuclear, then so is $\psi_\infty$.

\subsection{Absorbing Representations}

Let $A$ and $B$ be C$^*$-algebras such that $A$ is separable and $B$ is $\sigma$-unital.  Given two representations $\varphi, \psi : A \rightarrow M(B \otimes \mathcal{K})$, write $\varphi \sim \psi$ if there is a sequence of unitaries $(u_n)_{n=1}^\infty \subseteq M(B \otimes \mathcal{K})$ such that
\begin{enumerate}
  \item $\| \varphi(a) - u_n \psi(a) u_n^* \| \rightarrow 0$ as $n \rightarrow \infty$ and
  \item $\varphi(a) - u_n \psi(a) u_n^* \in B \otimes \mathcal{K}$ for all $n \geq 1$
\end{enumerate}
for all $a \in A$.

\begin{definition}\label{defn:Absorbing}
Suppose $A$ and $B$ are C$^*$-algebras such that $A$ is separable and $B$ is $\sigma$-unital.  A representation $\varphi : A \rightarrow M(B \otimes \mathcal{K})$ is called (\emph{unitally}) \emph{nuclearly absorbing} if for all (unital) weakly nuclear representations $\theta : A \rightarrow M(B \otimes \mathcal{K})$, $\varphi \oplus \theta \sim \varphi$.
\end{definition}

Consider the special case when $B = \mathbb{C}$.  All representations $A \rightarrow M(\mathcal{K}) \cong \mathcal{B}(\ell^2(\mathbb{N}))$ are weakly nuclear as $\mathcal{K}$ is nuclear.  Now, Voiculescu's representation theorem, as stated in Theorem II.5.8 of \cite{Davidson:C*Book} for example, is the statement that a unital representation $\varphi : A \rightarrow M(\mathcal{K})$ is unitally nuclearly absorbing if, and only if, $\varphi$ is faithful and $\varphi(A) \cap \mathcal{K} = 0$.  There is a far reaching generalization of this result for nuclearly absorbing representations due to the work of Elliott and Kucerovsky in \cite{ElliottKucerovsky} (see also \cite{KucerovskyNg} and \cite{Gabe:AbsorbingExtensions}).

A $\sigma$-unital C$^*$-algebra $B$ has the \emph{corona factorization property} if for all projections $p \in M(B \otimes \mathcal{K})$, $1 \precsim p \oplus p$ implies $1 \precsim p$ (see \cite{KucerovskyNg}).  The corona factorization property is a very weak regularity property of C$^*$-algebras.  See \cite{KirchbergRordam:CentralSequenceAlgebras, OrtegaPereraRordamI, OrtegaPereraRordamII} and the references within for several examples and connections to other regularity properties.  In this paper, the examples of interest will be separable C$^*$-subalgebras of the trace-kernel ideal $J_B$ associated to an appropriate C$^*$-algebra $B$ as discussed in the introduction and introduced formally in Section \ref{sec:TraceKernel}.

When $B$ is unital, a $^*$-homomorphism $\varphi : A \rightarrow B$ is called \emph{unitizably full} if the unitization $\varphi^\dag : A^\dag \rightarrow B$ is full.  Note that when $A$ and $B$ are unital C$^*$-algebras, a $^*$-homomorphism $\varphi : A \rightarrow B$ is unitizably full if, and only if, $\varphi$ is full and $1_B - \varphi(1_A)$ is full.

\begin{theorem}\label{thm:ElliottKucerovskyGabe}
If $A$ is a separable $C^*$-algebra and $B$ is a $\sigma$-unital $C^*$-algebra with the corona factorization property, then every unitizably full representation $A \rightarrow M(B \otimes \mathcal{K})$ is nuclearly absorbing.
\end{theorem}

\begin{proof}
After adding units, it is enough to show that every unital, full representation $\varphi : A \rightarrow M(B \otimes \mathcal{K})$ is unitally nuclearly absorbing.  Let $q : M(B \otimes \mathcal{K}) \rightarrow M(B \otimes \mathcal{K}) / (B \otimes \mathcal{K})$ denote the quotient map.  By \cite{KucerovskyNg}, $q \varphi$ is a purely large extension, and hence, by the main result of \cite{ElliottKucerovsky}, $q \varphi$ is unitally nuclearly absorbing as an extension.

For any unital, weakly nuclear representation $\psi : A \rightarrow M(B \otimes \mathcal{K})$, there is a unitary $u \in M(B \otimes \mathcal{K})$ such that
\[ u(\varphi(a) \oplus \psi_\infty(a))u^* - \psi_\infty(a) \in B \otimes \mathcal{K} \]
for all $a \in A$.  By the equivalence of (iv) and (v) in Theorem 3.4 of \cite{GabeRuiz:AbsorbingCones}, $\varphi \oplus \psi_\infty \sim \varphi$.  As $\psi_\infty \oplus \psi$ and $\psi_\infty$ are unitarily equivalent,
\[ \varphi \oplus \psi \sim \varphi \oplus \psi_\infty \oplus \psi \sim \varphi \oplus \psi_\infty \sim \varphi, \]
so $\varphi$ is unitally nuclearly absorbing.
\end{proof}

Given two representations $\varphi, \psi : A \rightarrow M(B \otimes \mathcal{K})$, write $\varphi \sim_{\mathrm{asymp}} \psi$ if there is a norm continuous family $(u_t)_{t \geq 0} \subseteq M(B \otimes \mathcal{K})$ of unitaries such that
\begin{enumerate}
  \item $\|\varphi(a) - u_t \psi(a) u_t^* \| \rightarrow 0$ as $t \rightarrow \infty$, and
  \item $\varphi(a) - u_t \psi(a) u_t^* \in B \otimes \mathcal{K}$ for all $t \geq 0$
\end{enumerate}
for all $a \in A$.

The following folklore result shows nuclearly absorbing representations also satisfy a stronger asymptotic absorption condition.

\begin{proposition}\label{prop:AsymptoticAbsorbing}
Suppose $A$ and $B$ are $C^*$-algebras such that $A$ is separable and $B$ is $\sigma$-unital.  If $\varphi : A \rightarrow M(B \otimes \mathcal{K})$ is a nuclearly absorbing representation and $\psi : A \rightarrow M(B \otimes \mathcal{K})$ is a weakly nuclear representation, then $\varphi \oplus \psi \sim_\mathrm{asymp} \varphi$.
\end{proposition}

\begin{proof}
As $\varphi$ is nuclearly absorbing, $\varphi \oplus \psi_\infty \sim \varphi$, and hence $\varphi \oplus \psi_\infty \sim_\mathrm{asymp} \varphi$ by the equivalence of (v) and (vi) in Theorem 3.4 of \cite{GabeRuiz:AbsorbingCones}.  Now,
\[ \varphi \oplus \psi \sim_\mathrm{asymp} \varphi \oplus \psi_\infty \oplus \psi \sim_\mathrm{asymp} \varphi \oplus \psi_\infty \sim_\mathrm{asymp} \varphi. \]
\end{proof}

\subsection{Destabilizing $KK$-theory}

Following Dadarlat and Eilers in \cite{DadarlatEilers:CuntzPairs}, two representations $\varphi, \psi : A \rightarrow M(B \otimes \mathcal{K})$ are \emph{properly asymptotically unitarily equivalent}, written $\varphi \approxeq \psi$, if there is a norm continuous family of unitaries $(u_t)_{t \geq 0}$ in $B \otimes \mathcal{K} + \mathbb{C} 1_{M(B \otimes \mathcal{K})} \subseteq M(B \otimes \mathcal{K})$ such that
\begin{enumerate}
  \item $\| \varphi(a) - u_t \psi(a) u_t^* \| \rightarrow 0$ as $t \rightarrow \infty$ and
  \item $\varphi(a) - u_t \psi(a) u_t^* \in B \otimes \mathcal{K}$ for all $t \geq 0$
\end{enumerate}
for all $a \in A$.

The word \emph{proper} reflects that the path of unitaries is taken from the minimal unitization of $B \otimes \mathcal{K}$ instead of the multiplier algebra.  This is a subtle, but very critical, difference with the relation $\sim_\mathrm{asymp}$ above.  The following result, which is essentially due to Dadarlat and Eilers in \cite{DadarlatEilers:CuntzPairs}, shows the relevance of proper asymptotic equivalence in $KK$-theory.  This is in stark contrast with Proposition \ref{prop:AsymptoticAbsorbing} above, which shows the relation $\sim_\mathrm{asymp}$ is a rather weak equivalence relation on representations.

\begin{theorem}\label{thm:DadarlatEilers}
Suppose $A$ is a separable $C^*$-algebra, $B$ is a $\sigma$-unital $C^*$-algebra, and $(\varphi, \psi) \in \mathcal{E}_\mathrm{nuc}(A, B)$ is a Cuntz pair.  The following are equivalent:
\begin{enumerate}
  \item $[\varphi, \psi] = 0 \in KK_\mathrm{nuc}(A, B)$;
  \item there is a weakly nuclear representation $\theta : A \rightarrow M(B \otimes \mathcal{K})$ such that $\varphi \oplus \theta \approxeq \psi \oplus \theta$;
  \item for any weakly nuclear, nuclearly absorbing representation $\theta : A \rightarrow M(B \otimes \mathcal{K})$, $\varphi \oplus \theta \approxeq \psi \oplus \theta$.
\end{enumerate}
\end{theorem}

\begin{proof}
By Theorem 3.10 of \cite{DadarlatEilers:CuntzPairs}, (1) and (2) are equivalent to
\begin{enumerate}
  \item[(3')] for any weakly nuclear, nuclearly absorbing representation $\theta : A \rightarrow M(B \otimes \mathcal{K})$, $\varphi \oplus \theta_\infty \approxeq \psi \oplus \theta_\infty$.
\end{enumerate}
Hence it suffices to show (3) is equivalent to (3').  If $\theta$ is weakly nuclear and nuclearly absorbing, then $\theta_\infty$ is weakly nuclear, and hence $\theta \oplus \theta_\infty \sim_\mathrm{asymp} \theta$ by Proposition \ref{prop:AsymptoticAbsorbing}.  So $\theta_\infty \sim_\mathrm{asymp} \theta$, and the result follows from Lemma 3.4 in \cite{DadarlatEilers:CuntzPairs}.
\end{proof}

The next two results in this section will allow us to control the stabilizations appearing in the previous theorem when all representations considered are assumed to be nuclearly absorbing.  These results form a critical part of the classification results in Section \ref{sec:UltraproductClassification} and play a role analogous to that of the stable uniqueness theorem used in earlier classification results.

\begin{proposition}\label{prop:PrescribedCuntzPair}
Suppose $A$ and $B$ are $C^*$-algebras such that $A$ is separable and $B$ is $\sigma$-unital.  If $x \in KK_\mathrm{nuc}(A, B)$ and $\psi : A \rightarrow M(B \otimes \mathcal{K})$ is a weakly nuclear, nuclearly absorbing representation, then there is a weakly nuclear, nuclearly absorbing representation $\varphi : A \rightarrow M(B \otimes \mathcal{K})$ such that $(\varphi, \psi)$ is a Cuntz pair in $\mathcal{E}_\mathrm{nuc}(A, B)$ and $x = [\varphi, \psi]$.
\end{proposition}

\begin{proof}
Let $x \in KK_\mathrm{nuc}(A, B)$ and $\psi$ be given.  There is a Cuntz pair $(\theta, \rho)$ in $\mathcal{E}_\mathrm{nuc}(A, B)$ such that $x = [\theta, \rho]$.  By Proposition \ref{prop:AsymptoticAbsorbing}, there is a norm continuous family $(u_t)_{t \geq 0}$ of unitaries in $M(B \otimes \mathcal{K})$ such that
\begin{enumerate}
  \item $\| \psi(a) - u_t(\rho \oplus \psi)(a)u_t^* \| \rightarrow 0$ as $t \rightarrow \infty$ and
  \item $\psi(a) - u_t(\rho \oplus \psi)(a)u_t^* \in B \otimes \mathcal{K}$ for all $t \geq 0$
\end{enumerate}
for all $a \in A$.

For each $t \geq 0$ and $a \in A$, $u_t(\rho \oplus \psi)(a) u_t^*$ and $u_0 (\rho \oplus \psi)(a) u_0^*$ differ by an element of $B \otimes \mathcal{K}$ as both elements differ from $\psi(a)$ by an element of $B \otimes \mathcal{K}$.  As $(\theta, \rho)$ is a Cuntz pair, we have
\[ u_0((\theta \oplus \psi)(a) - (\rho \oplus \psi)(a))u_0^* \in B \otimes \mathcal{K} \]
for all $a \in A$.  Hence $(\mathrm{ad}(u_0)(\theta \oplus \psi), \mathrm{ad}(u_t) (\rho \oplus \psi))$ is a Cuntz pair for all $t \geq 0$ and defines a homotopy\footnote{Here the homotopy is defined on $[0, \infty]$ in place of $[0, 1]$.} between the Cuntz pairs $(\mathrm{ad}(u_0)(\theta \oplus \psi), \mathrm{ad}(u_0)(\rho \oplus \psi))$ and $(\mathrm{ad}(u_0)(\theta \oplus \psi), \psi)$.  Now,
\[ x = [\theta \oplus \psi, \rho \oplus \psi] = [\mathrm{ad}(u_0)(\theta \oplus \psi), \mathrm{ad}(u_0)(\rho \oplus \psi)] = [\mathrm{ad}(u_0)(\theta \oplus \psi), \psi]. \]
To complete the proof, define $\varphi = \mathrm{ad}(u_0)(\theta \oplus \psi)$.
\end{proof}

In the uniqueness portion of Theorem D, we will need a way to relate two weakly nuclear representations $\varphi, \psi : A \rightarrow M(B \otimes \mathcal{K})$ when $(\varphi, \psi)$ forms a Cuntz pair and $[\varphi, \psi] = 0$ in $KK_\mathrm{nuc}(A, B)$.  Ideally, if $\varphi$ and $\psi$ are nuclearly absorbing, then  $[\varphi, \psi] = 0$ would imply $\varphi \approxeq \psi$.  This is known to be the case when $B = \mathcal{K}$ (see Theorem 3.12 in \cite{DadarlatEilers:CuntzPairs}) but is not known in general.  The following technical variation will be sufficient for our purposes.  A stronger result of this form will appear in forthcoming work of the author with Carri\'{o}n, Gabe, Tikuisis, and White.

\begin{proposition}\label{prop:KKUniqueness}
Suppose $A$ is a separable $C^*$-algebra, $E$ is a separable, unital, $\mathcal{Q}$-stable $C^*$-algebra, and $I \subseteq E$ is an ideal such that $I \otimes \mathcal{K} \cong I$.  Suppose $\varphi, \psi : A \rightarrow E$ are nuclear $^*$-homomorphisms such that $\varphi(a) - \psi(a) \in I$ for all $a \in A$. Let $\lambda : E \rightarrow M(I)$ be the canonical $^*$-homomorphism and note that $(\lambda \varphi, \lambda \psi) \in \mathcal{E}_\mathrm{nuc}(A, I)$.

If $\lambda \varphi$ and $\lambda \psi$ are nuclearly absorbing representations and $[\lambda \varphi, \lambda \psi] = 0 \in KK_\mathrm{nuc}(A, I)$, then there is a sequence of unitaries $(u_n)_{n=1}^\infty \subseteq E$ such that
\[ \| \varphi(a) - u_n \psi(a)u_n^* \| \rightarrow 0 \]
for all $a \in A$.
\end{proposition}

\begin{proof}
Since $[\lambda \varphi, \lambda \psi] = 0 \in KK_\mathrm{nuc}(A, I)$, $\lambda \varphi$ is nuclearly absorbing, and $\lambda \psi$ is weakly nuclear, we have $\lambda \varphi \oplus \lambda \varphi \approxeq \lambda \psi \oplus \lambda \varphi$ by Theorem \ref{thm:DadarlatEilers}.  Reversing the roles of $\varphi$ and $\psi$, we have $\lambda \varphi \oplus \lambda \psi \approxeq \lambda \psi \oplus \lambda \psi$.  In particular, there is a sequence $(u_n)_{n=1}^\infty \subseteq \mathbb{M}_2(I + \mathbb{C}1_{M(I)})$ such that
\[ \| (\lambda \varphi \oplus \lambda \varphi)(a) - u_n (\lambda \psi \oplus \lambda \psi)(a) u_n^* \| \rightarrow 0 \]
for all $a \in A$.  As the restriction of $\lambda$ to $\mathbb{M}_2(I + \mathbb{C}1_E)$ is injective, it follows that $\varphi \oplus \varphi$ and $\psi \oplus \psi$ are approximately unitarily equivalent as $^*$-homomorphisms $A \rightarrow \mathbb{M}_2(E)$.

As there is a unital embedding $\mathbb{M}_2 \hookrightarrow \mathcal{Q}$, $\varphi \otimes 1_\mathcal{Q}$ and $\psi \otimes 1_\mathcal{Q}$ are approximately unitarily equivalent as $^*$-homomorphisms $A \rightarrow E \otimes \mathcal{Q}$.  As $E$ is $\mathcal{Q}$-stable and $\mathcal{Q}$ is strongly self-absorbing, there is, by Remark 2.7 in \cite{TomsWinter:SSA}, a sequence $\theta_n : E \otimes \mathcal{Q} \rightarrow E$ of unital $^*$-homomorphisms such that $\theta_n(x \otimes 1_\mathcal{Q}) \rightarrow x$ for all $x \in E$.  An $\varepsilon/3$ argument implies that $\varphi$ and $\psi$ are approximately unitarily equivalent as $^*$-homomorphisms $A \rightarrow E$.
\end{proof}

\section{The Trace-Kernel Extension}\label{sec:TraceKernel}

Let $B$ be a simple, unital C$^*$-algebra with a unique trace $\tau_B$ and define the \emph{2-norm} on $B$ by $\|b\|_2 = \tau_B(b^*b)^{1/2}$ for all $b \in B$.  Let $\ell^\infty(B)$ denote the C$^*$-algebra of all bounded sequences in $B$ and, for a free ultrafilter $\omega$ on the natural numbers, define
\begin{align*}
   B_\omega &:= \ell^\infty(B) / \{ b \in \ell^\infty(B) : \lim_{n \rightarrow \omega} \|b_n\| = 0 \} \quad \text{and} \\
   B^\omega &:= \ell^\infty(B) / \{ b \in \ell^\infty(B) : \lim_{n \rightarrow \omega} \|b_n\|_2 = 0 \}.
\end{align*}

Since $\tau_B$ is contractive, $\|b\|_2 \leq \|b\|$ for all $b \in B$, and hence there is a natural extension
\[ \begin{tikzcd} 0 \arrow{r} & J_B \arrow{r}{j_B} & B_\omega \arrow{r}{q_B} & B^\omega \arrow{r} & 0 \end{tikzcd} \]
where $q_B$ is the canonical quotient map, $J_B = \ker(q_B)$, and $j_B$ is the inclusion map.  The C$^*$-algebra $J_B$ is referred to as the \emph{trace-kernel ideal} associated to $B$, and the extension is called the \emph{trace-kernel extension} associated to $B$.

The trace-kernel extension has been extensively studied in connection with the Toms-Winter conjecture (with a modified definition when $B$ has more than one trace); see \cite{BBSTWW, KirchbergRordam:TomsWinterConj, MatuiSato:ZStable, MatuiSato:DecompositionRank, Sato:TomsWinter, SatoWhiteWinter, TomsWhiteWinter} for example.  In the case $B = \mathcal{Q}$, this extension also has a crucial role in the proof of the quasidiagonality theorem (Theorem \ref{thm:TWW} above), where the basic strategy is to lift a trace-preserving $^*$-homomorphism into $\mathcal{Q}^\omega \cong \mathcal{R}^\omega$ along the quotient map $q_\mathcal{Q}$ to obtain a trace-preserving $^*$-homomorphism into $\mathcal{Q}_\omega$.   This was made more explicit in \cite{Schafhauser:Crelle} where the lift was obtained through extension theoretic methods.

Many of the properties of $J_\mathcal{Q}$ proved in \cite{Schafhauser:Crelle} also hold for $J_B$ for much more general C$^*$-algebras $B$ (Proposition \ref{prop:TraceKernelExtension} below).  The following definition is taken from \cite{Schafhauser:Crelle}.

\begin{definition}\label{defn:AdmissibleKernel}
A C$^*$-algebra $I$ is called an \emph{admissible kernel} if $I$ has real rank zero and stable rank one, $K_0(I)$ is divisible, $K_1(I) = 0$, the Murray-von Neumann semigroup $V(I)$ is almost unperforated, and every projection in $I \otimes \mathcal{K}$ is Murray-von Neumann equivalent to a projection in $I$.
\end{definition}

The following result collects the key properties of the trace-kernel extension needed in the next section.

\begin{proposition}\label{prop:TraceKernelExtension}
If $B$ is a simple, unital, $\mathcal{Q}$-stable $C^*$-algebra with a unique trace $\tau_B$ such that every quasitrace on $B$ is a trace and $K_1(B) = 0$, then
\begin{enumerate}
  \item $B^\omega$ is a $\mathrm{II}_1$-factor,
  \item $B_\omega$ has real rank zero and stable rank one, has a unique trace $\tau_{B_\omega}$, has strict comparison of positive elements with respect to the trace, is separably $\mathcal{Q}$-stable, and has trivial $K_1$-group, and
  \item $J_B$ is an admissible kernel.
\end{enumerate}
\end{proposition}

\begin{proof}
For (1), as $B$ has a unique trace $\tau_B$, $\pi_{\tau_B}(B)''$ is a $\mathrm{II}_1$-factor (see Theorem 6.7.4 and Corollary 6.8.5 in \cite{Dixmier:C*algebras}), and hence so is the tracial ultrapower $(\pi_{\tau_B}(B)'')^\omega$.  By Remark 4.7 in \cite{KirchbergRordam:TomsWinterConj}, $B^\omega \cong B_\omega / J_B \cong (\pi_{\tau_B}(B)'')^\omega$.

For (2), the C$^*$-algebra $B$ has real rank zero by Theorem 7.2 of \cite{Rordam:UHFStable2}, stable rank one by Corollary 6.6 in \cite{Rordam:UHFStable}, and strict comparison by Theorem 5.2 of \cite{Rordam:UHFStable2}.  All three properties are known to be preserved by ultraproducts; for strict comparison, this follows from Lemma 1.23 in \cite{BBSTWW}, and for real rank zero and stable rank one, see the proof of Proposition 3.2 in \cite{Schafhauser:Crelle} for example.  As $B$ is $\mathcal{Q}$-stable, $B_\omega$ is separably $\mathcal{Q}$-stable by Proposition \ref{prop:DStablePermanence}.  By Theorem 8 in \cite{Ozawa:DixmierApproximation}, the unique trace on $B_\omega$ is the trace $\tau_{B_\omega}$ induced by $\tau_B$.

As $B_\omega$ has stable rank one, to show $K_1(B_\omega) = 0$, it suffices by Theorem 2.10 of \cite{Rieffel:NCToriUnitaries} to show that the unitary group of $B_\omega$ is path connected.  Let $u$ be a unitary in $B_\omega$ and fix a sequence of unitaries $(u_n)_{n=1}^\infty$ in $B$ lifting $u$.  As $K_1(B) = 0$ and $B$ has stable rank one, another application of Theorem 2.10 in \cite{Rieffel:NCToriUnitaries} shows $u_n$ is in the path component of the identity for each $n \geq 1$.  Since $B$ has real rank zero, Theorem 5 of \cite{Lin:BrownPedersen} implies there is a unitary $v_n \in B$ with finite spectrum such that $\| u_n - v_n \| < 1/n$.  Write $v_n = e^{i h_n}$ for a self-adjoint $h_n \in B$ with $\|h_n\| \leq \pi$ and let $h$ denote the self-adjoint element of $B_\omega$ determined by the sequence $(h_n)_{n=1}^\infty$.  Then $u = e^{ih}$, and hence $u$ is in the path component of the identity in the unitary group of $B_\omega$.

For (3), as $B_\omega$ has real rank zero and stable rank one, so does $J_B$ as both properties pass of ideals by Corollary 2.8 in \cite{BrownPedersen:RealRankZero} and Theorem 4.3 in \cite{Rieffel:StableRank}, respectively.  If $d \geq 1$ is an integer and $p \in \mathbb{M}_d(J_B)$ is a projection, then
\[ (\tau_{B_\omega} \otimes \mathrm{Tr}_{\mathbb{M}_d})(p) = 0 < 1 = (\tau_{B_\omega} \otimes \mathrm{Tr}_{\mathbb{M}_d})(1_{B_\omega} \oplus 0_{B_\omega}^{\oplus (d-1)}) \]
where $\mathrm{Tr}_{\mathbb{M}_d}$ is the tracial functional on $\mathbb{M}_d$ normalized at a rank one projection.  As $B_\omega$ has strict comparison, there is a partial isometry $v \in \mathbb{M}_d(B_\omega)$ such that $v^*v = p$ and $vv^* \leq 1_{B_\omega} \oplus 0_{B_\omega}^{\oplus (d-1)}$.  Since $v^*v \in \mathbb{M}_d(J_B)$, we have $vv^* \in \mathbb{M}_d(J_B)$, and hence $vv^* = q \oplus 0_{B_\omega}^{\oplus (d-1)}$ for some projection $q \in J_B$.  So every projection in $J_B \otimes \mathcal{K}$ is Murray-von Neumann equivalent to a projection in $J_B$.

Since $B_\omega$ is separably $\mathcal{Q}$-stable, so is $J_B$ by Proposition \ref{prop:DStablePermanence}.  Hence there is an increasing net $J_i$ of a separable, $\mathcal{Q}$-stable C$^*$-subalgebras of $J_B$ with dense union.  It is well known that, since $J_i$ is $\mathcal{Q}$-stable, $V(J_i)$ is unperforated and $K_0(J_i)$ is divisible; this can be shown using the continuity of $V(-)$ and $K_0(-)$ and writing $J_i \otimes \mathcal{Q}$ as the inductive limit of the $^*$-homomorphisms $\mathbb{M}_k(J_i) \rightarrow \mathbb{M}_\ell(J_i)$ where $k$ divides $\ell$ and the maps are given by the diagonal embeddings with multiplicity $\ell / k$.  Using the continuity of the functors $V(-)$ and $K_0(-)$ again, it follows that $V(J_B)$ is almost unperforated and $K_0(J_B)$ is divisible as both properties are preserved by inductive limits.

It remains to show $K_1(J_B) = 0$.  Consider the exact sequence
\[ \begin{tikzcd} K_0(B_\omega) \arrow{r}{K_0(q_B)} & K_0(B^\omega) \arrow{r}{\partial_0} & K_1(J_B) \arrow{r}{K_1(j_B)} & K_1(B_\omega) \end{tikzcd} \]
induced by the trace-kernel extension.  As $K_1(B_\omega) = 0$, it suffices to show $K_0(q_B)$ is surjective.  Suppose $t \in [0, 1]$ and write $t = \lim_{n \rightarrow \omega} t_n$ with $t_n \in \mathbb{Q} \cap [0, 1]$.  As $B$ is unital and $\mathcal{Q}$-stable, there is a unital embedding $\mathcal{Q} \hookrightarrow B$.  Hence there is a sequence of projections $(p_n)_{n=1}^\infty \subseteq B$ with $\tau_B(p_n) = t_n$.  If $p$ denotes the projection in $B_\omega$ defined by the sequence $p_n$, then $\tau_{B_\omega}(p) = t$.  Hence the group homomorphism $\hat{\tau}_{B_\omega} : K_0(B_\omega) \rightarrow \mathbb{R}$ is surjective, and since $B^\omega$ is a $\mathrm{II}_1$-factor, the group homomorphism $\hat{\tau}_{B^\omega} : K_0(B^\omega) \rightarrow \mathbb{R}$ is an isomorphism.  Now, $\hat{\tau}_{B_\omega} = \hat{\tau}_{B^\omega} K_0(q_B)$, so $K_0(q_B)$ is surjective, and hence $K_1(J_B) = 0$.
\end{proof}

The relevant properties of admissible kernels are collected in the following result.

\begin{proposition}\label{prop:AdmissibleKernel}\hfill
\begin{enumerate}
  \item The property of being an admissible kernel is separably inheritable.
  \item If $I$ is an admissible kernel, then $\mathbb{M}_n(I)$ is an admissible kernel for all $n \geq 1$.
  \item If $I$ is a separable admissible kernel, then $I$ is stable and has the corona factorization property.
  \item If $A$ is a separable $C^*$-algebra satisfying the UCT and $I$ is a separable admissible kernel, then the canonical homomorphism $KK_\mathrm{nuc}(A, I) \rightarrow \mathrm{Hom}_{\mathbb{Z}}(K_0(A), K_0(I))$ is an isomorphism.
\end{enumerate}
\end{proposition}

\begin{proof}
(1) is Proposition 4.1 in \cite{Schafhauser:Crelle}.  In (2), the only non-trivial claim is that real rank zero and stable rank one are preserved by taking matrix algebras; this follows from Theorem 2.10 in \cite{BrownPedersen:RealRankZero} and Theorem 3.3 in \cite{Rieffel:StableRank}, respectively.  The first two paragraphs of the proof of Theorem 2.1 in \cite{Schafhauser:Crelle} show (3).  Since $K_1(I) = 0$ and $K_0(I)$ is divisible, it is an immediate consequence of the UCT that the natural map $KK(A, I) \rightarrow \mathrm{Hom}_{\mathbb{Z}}(K_0(A), K_0(I))$ is an isomorphism.  By Theorem 4.1 and Proposition 7.1 in \cite{RosenbergSchochet}, $A$ is $KK$-equivalent to a commutative C$^*$-algebra, and hence, by Propositions 3.2 and 3.3 in \cite{Skandalis:KKnuc}, the canonical map $KK_\mathrm{nuc}(A, I) \rightarrow KK(A, I)$ is an isomorphism.
\end{proof}

\section{Classification of $^*$-homomorphisms into ultrapowers}\label{sec:UltraproductClassification}

The goal of this section is to produce classification results for unital, full, nuclear $^*$-homomorphisms from a separable, exact C$^*$-algebra $A$ satisfying the UCT into an ultra\-power of a suitably well-behaved codomain $B$ as outlined in the introduction.  The following pullback lemma, shown to me by Jamie Gabe, will be used to control the range of a representation $A \rightarrow M(J_B)$ in the proof of Proposition \ref{prop:UltraproductExistence}.  This lemma is implicitly contained in the proof Theorem 5.1 in \cite{Schafhauser:Crelle} and was explicitly stated in a slightly different form in an earlier version available on the arXiv.  As the result does not appear in the published version of \cite{Schafhauser:Crelle}, the statement and proof are reproduced here.

\begin{lemma}\label{lemma:pullback}
Consider a commuting diagram
\[ \begin{tikzcd} 0 \arrow{r} & I \arrow{r}{j_1} \arrow[equals]{d} & P \arrow{r}{\alpha_1} \arrow{d}{\alpha_2} & B_1 \arrow{r} \arrow{d}{\beta_1} & 0 \\
   0 \arrow{r} & I \arrow{r}{j_2} & B_2 \arrow{r}{\beta_2} & D \arrow{r} & 0 \end{tikzcd} \]
of $C^*$-algebras with exact rows.  If $A$ is a $C^*$-algebra and $\varphi_i : A \rightarrow B_i$ are $^*$-homomorphisms for $i = 1, 2$ such that $\beta_1 \varphi_1 = \beta_2 \varphi_2$, then there is a unique $^*$-homomorphism $\varphi : A \rightarrow P$ such that $\alpha_i \varphi = \varphi_i$ for $i = 1, 2$.  Moreover, $\varphi$ is nuclear if, and only if, $\varphi_1$ and $\varphi_2$ are nuclear.
\end{lemma}

\begin{proof}
Consider the pullback C$^*$-algebra $Q = \{ (b_1, b_2) \in B_1 \oplus B_2 : \beta_1(b_1) = \beta_2(b_2) \}$ and define $\pi : P \rightarrow Q$ by $\pi(p) = (\alpha_1(p), \alpha_2(p))$.  A diagram chase shows $\pi$ is an isomorphism.  For existence, the $^*$-homomorphism $\varphi$ is given by $\varphi(a) = \pi^{-1}(\varphi_1(a), \varphi_2(a))$, and uniqueness reduces to the statement $\ker(\alpha_1) \cap \ker(\alpha_2) = \ker(\pi) = 0$.

If $\varphi$ is nuclear, then $\varphi_i = \alpha_i \varphi$ is nuclear for $i = 1, 2$.  Conversely, suppose $\varphi_1$ and $\varphi_2$ are nuclear.  Fix a C$^*$-algebra $C$ and let $\rho : A \otimes_\mathrm{max} C \rightarrow A \otimes_\mathrm{min} C$ be the canonical $^*$-homomorphism.  As maximal tensor products preserve exact sequences, there is a commuting diagram
\[ \begin{tikzcd} 0 \arrow{r} & I \otimes_\mathrm{max} C \arrow{r}{j_1 \otimes \mathrm{id}} \arrow[equals]{d} & P \otimes_\mathrm{max} C \arrow{r}{\alpha_1 \otimes \mathrm{id}} \arrow{d}{\alpha_2 \otimes \mathrm{id}} & B_1 \otimes_\mathrm{max} C \arrow{r} \arrow{d}{\beta_1 \otimes \mathrm{id}} & 0 \\
   0 \arrow{r} & I \otimes_\mathrm{max} C \arrow{r}{j_2\otimes \mathrm{id}} & B_2 \otimes_\mathrm{max} C \arrow{r}{\beta_2 \otimes \mathrm{id}} & D \otimes_\mathrm{max} C \arrow{r} & 0 \end{tikzcd} \]
with exact rows.  As $\varphi_i$ is nuclear, there is a $^*$-homomorphism $\psi_i : A \otimes_\mathrm{min} C \rightarrow B_i \otimes_\mathrm{max} C$ such that $\varphi_i \otimes_{\mathrm{max}} \mathrm{id}_C = \psi_i \rho$ for each $i = 1, 2$ by Corollary 3.8.8 of \cite{BrownOzawa}.

Applying the first part of the lemma to the maps $\psi_1$ and $\psi_2$, there is a $^*$-homomorphism $\psi : A \otimes_\mathrm{min} C \rightarrow P \otimes_\mathrm{max} C$ such that $(\alpha_i \otimes_\mathrm{max} \mathrm{id}_C) \psi = \psi_i$.  Since
\[ (\alpha_i \otimes_\mathrm{max} \mathrm{id}_C) \psi \rho = \varphi_i \otimes_\mathrm{max} \mathrm{id}_C = (\alpha_i \otimes_\mathrm{max} \mathrm{id}_C) (\varphi \otimes_\mathrm{max} \mathrm{id}_C), \]
the uniqueness portion of the first part of the lemma implies $\psi \rho = \varphi \otimes_\mathrm{max} \mathrm{id}_C$.  So $\varphi \otimes_{\max} \mathrm{id}_C$ factors through $A \otimes_\mathrm{min} C$.  As $C$ was arbitrary, $\varphi$ is nuclear by Corollary 3.8.8 of \cite{BrownOzawa}.
\end{proof}

The next two propositions provide a first approximation to the existence and uniqueness results in Theorem D and its technical refinement given in Theorem \ref{thm:Classification}.

For any C$^*$-algebra $C$, $\iota_2 : C \rightarrow \mathbb{M}_2(C)$ denotes the inclusion into the $(1,1)$-corner, and for any $^*$-homomorphism $f : C_1 \rightarrow C_2$ between C$^*$-algebras $C_1$ and $C_2$, the induced $^*$-homomorphism $\mathbb{M}_2(C_1) \rightarrow \mathbb{M}_2(C_2)$ is still denoted by $f$.

\begin{proposition}\label{prop:UltraproductExistence}
Suppose $A$ is a separable, unital, exact $C^*$-algebra satisfying the UCT and $B$ is a simple, unital, $\mathcal{Q}$-stable $C^*$-algebra with a unique
trace $\tau_B$ such that every quasitrace on $B$ is a trace and $K_1(B) = 0$.

If $\tau_A$ is a faithful, amenable trace on $A$ and $\sigma : K_0(A) \rightarrow K_0(B_\omega)$ is a group homomorphism such that $\hat{\tau}_{B_\omega} \sigma = \hat{\tau}_A$ and $\sigma([1_A]) = [1_{B_\omega}]$, then there is a unital, full, nuclear $^*$-homomorphism $\varphi : A \rightarrow B_\omega$ such that $K_0(\varphi) = \sigma$ and $\tau_{B_\omega} \varphi = \tau_A$.
\end{proposition}

\begin{proof}
As $B$ is $\mathcal{Q}$-stable, there is a unital embedding $\mathcal{Q} \rightarrow B$ which is necessarily trace-preserving by the uniqueness of the trace on $\mathcal{Q}$.  This induces a unital, trace-preserving embedding $\mathcal{Q}_\omega \rightarrow B_\omega$.  Composing this embedding with the $^*$-homomorphism $A \rightarrow \mathcal{Q}_\omega$ given by Theorem \ref{thm:TWW} yields a unital, full, nuclear $^*$-homomorphism $\psi : A \rightarrow B_\omega$ such that $\tau_{B_\omega} \psi = \tau_A$.

Note that $\hat{\tau}_{B^\omega} K_0(q_B \psi) = \hat{\tau}_{B^\omega} K_0(q_B) \sigma$.  As $B^\omega$ is a $\mathrm{II}_1$-factor by Proposition \ref{prop:AdmissibleKernel}, $\hat{\tau}_{B^\omega}$ is an isomorphism, and hence $K_0(q_B \psi) = K_0(q_B) \sigma$.  So the image of $\sigma - K_0(\psi)$ is contained in $\ker(K_0(q_B)) = \mathrm{im}(K_0(j_B))$.  Using again that $B^\omega$ is a $\mathrm{II}_1$-factor, $K_1(B^\omega) = 0$, and hence $K_0(j_B)$ is injective.  So there is a group homomorphism $\kappa : K_0(A) \rightarrow K_0(J_B)$ such that $K_0(j_B) \kappa = \sigma - K_0(\psi)$.

By Proposition \ref{prop:SeparableCorestriction}, there is a separable C$^*$-subalgebra $E_0$ of $B_\omega$ containing $\psi(A)$ such that the corestriction of $\psi$ to $E_0$ is full and nuclear.  By Proposition \ref{prop:NonseparableHomGroups}, there is a separable C$^*$-subalgebra $I_0$ of $J_B$ such that $\kappa$
factors as the composition of a group homomorphism $\kappa_0 : K_0(A) \rightarrow K_0(I_0)$ and the group homomorphism $K_0(I_0) \rightarrow K_0(J_B)$ induced by the inclusion $I_0 \hookrightarrow J_B$.  As $J_B$ is an admissible kernel by Proposition \ref{prop:TraceKernelExtension} and being an admissible kernel is separably inheritable by Proposition \ref{prop:AdmissibleKernel}, Proposition \ref{prop:SeparableExtension} implies there are a separable admissible kernel $I \subseteq J_B$ containing $I_0$, a separable C$^*$-subalgebra $E \subseteq B_\omega$ containing $E_0$, and a separable C$^*$-subalgebra $D \subseteq B^\omega$ such that there is a homomorphism
\[ \begin{tikzcd}
   0 \arrow{r} & I \arrow{r}{\hat{\jmath}} \arrow{d}{\iota_I} & E \arrow{r}{\hat{q}} \arrow{d}{\iota_E} & D \arrow{r} \arrow{d}{\iota_D} & 0 \\
   0 \arrow{r} & J_B \arrow{r}{j_B} & B_\omega \arrow{r}{q_B} & B^\omega \arrow{r} & 0
\end{tikzcd} \]
of extensions where the vertical maps are the inclusions.  Let $\hat{\psi} : A \rightarrow E$ denote the corestriction of $\psi$ to $E$ and let $\hat{\kappa} : K_0(A) \rightarrow K_0(I)$ be the composition of $\kappa_0$ with the group homomorphism $K_0(I_0) \rightarrow K_0(I)$ induced by the inclusion $I_0 \hookrightarrow I$.

As $\hat{\psi} : A \rightarrow E$ is full and nuclear, $\iota_2 \hat{\psi} : A \rightarrow \mathbb{M}_2(E)$ is unitizably full and nuclear.  Let $\lambda : \mathbb{M}_2(E) \rightarrow M(\mathbb{M}_2(I))$ be the canonical $^*$-homomorphism.  By Proposition \ref{prop:AdmissibleKernel}, $\mathbb{M}_2(I)$ is a separable admissible kernel whence is stable and has the corona factorization property.  Note that $\lambda \iota_2 \hat{\psi}$ is unitizably full since $\lambda$ is unital and $\iota_2 \hat{\psi}$ is unitizably full.  Now by Theorem \ref{thm:ElliottKucerovskyGabe}, $\lambda \iota_2 \hat{\psi}$ is nuclearly absorbing.

As $\mathbb{M}_2(I)$ is a separable admissible kernel and $A$ satisfies the UCT, the canonical group homomorphism
\[ KK_\mathrm{nuc}(A, \mathbb{M}_2(I)) \rightarrow \mathrm{Hom}_\mathbb{Z}(K_0(A), K_0(\mathbb{M}_2(I))) \]
is an isomorphism by Proposition \ref{prop:AdmissibleKernel}.  Let $x \in KK_\mathrm{nuc}(A, \mathbb{M}_2(I))$ be a lift of $K_0(\iota_2) \hat{\kappa}$.  By Proposition \ref{prop:PrescribedCuntzPair}, there is a weakly nuclear representation $\theta : A \rightarrow M(\mathbb{M}_2(I))$ such that $(\theta, \lambda \iota_2 \hat{\psi})$ is a Cuntz pair in $\mathcal{E}_\mathrm{nuc}(A, \mathbb{M}_2(I))$ and $[\theta, \lambda \iota_2 \hat{\psi}] = x$.

As $A$ is exact and $\theta$ is weakly nuclear, Proposition 3.2 of \cite{Gabe:LiftingTheorems} implies that $\theta$ is nuclear.  Lemma \ref{lemma:pullback} applied to the diagram
\[ \begin{tikzcd}
  0 \arrow{r} & \mathbb{M}_2(I) \arrow{r}{\hat{\jmath}} \arrow[equals]{d} & \mathbb{M}_2(E) \arrow{r}{\hat{q}} \arrow{d}{\lambda} & \arrow{r} \mathbb{M}_2(D) \arrow{r} \arrow{d} & 0 \\ 0 \arrow{r} & \mathbb{M}_2(I) \arrow{r} & M(\mathbb{M}_2(I)) \arrow{r} & M(\mathbb{M}_2(I)) / \mathbb{M}_2(I) \arrow{r} & 0
\end{tikzcd} \]
implies there is a nuclear $^*$-homomorphism $\hat{\varphi}_2 : A \rightarrow \mathbb{M}_2(E)$ such that $\lambda \hat{\varphi}_2 = \theta$ and $\hat{q} \hat{\varphi}_2 = \hat{q} \iota_2 \hat{\psi}$.

Now in $KK_\mathrm{nuc}(A, \mathbb{M}_2(E))$, we have
\[ \hat{\jmath}_*(x) = \hat{\jmath}_*[\lambda \hat{\varphi}_2, \lambda \iota_2 \hat{\psi}] = [\hat{\varphi}_2] - [\iota_2 \hat{\psi}] \]
by Proposition \ref{prop:KasparovProduct}.  In particular, as $x$ induces the group homomorphism $K_0(\iota_2)\hat{\kappa}$,
\[ K_0(\hat{\varphi}_2) - K_0(\iota_2 \hat{\psi}) = K_0(\iota_2 \hat{\jmath}) \hat{\kappa}. \]
Let $\varphi_2 = \iota_E \hat{\varphi}_2 : A \rightarrow \mathbb{M}_2(B_\omega)$.  Then
\[ K_0(\varphi_2) - K_0(\iota_2 \psi) = K_0(\iota_2 j_B) \kappa = K_0(\iota_2) \sigma - K_0(\iota_2 \psi) \]
by the choice of $\kappa$, and hence $K_0(\varphi_2) = K_0(\iota_2) \sigma$.  In particular,
\[ [\varphi_2(1_A)] = K_0(\iota_2)(\sigma([1_A])) = [1_{B_\omega}] \in K_0(B_\omega). \]
As $B_\omega$ has stable rank one by Proposition \ref{prop:TraceKernelExtension}, $B_\omega$ has cancellation of projections by Proposition 6.5.1 of \cite{Blackadar:KTheory}, and there is a unitary $u \in \mathbb{M}_2(B_\omega)$ with $u \varphi_2(1_A) u^* = 1_{B_\omega} \oplus 0_{B_\omega}$.  Now, there is a unital $^*$-homomorphism $\varphi : A \rightarrow B_\omega$ such that $\iota_2 \varphi = \mathrm{ad}(u) \varphi_2$.

We claim $\varphi$ is the desired $^*$-homomorphism.  Note that
\[ K_0(\iota_2 \varphi) = K_0(\varphi_2) = K_0(\iota_2) \sigma \]
by the unitary invariance of $K_0$.  By the stability of $K_0$, the map $K_0(\iota_2)$ is an isomorphism, and hence $K_0(\varphi) = \sigma$.  By construction, $q_B \varphi_2 = q_B \iota_2 \psi$, so if $\tau_{\mathbb{M}_2(B_\omega)}$ is the trace on $\mathbb{M}_2(B_\omega)$ induced by $\tau_{B_\omega}$, then
\[ \tau_A = \tau_{B_\omega} \psi = 2 \tau_{\mathbb{M}_2(B_\omega)} \varphi_2 = 2 \tau_{\mathbb{M}_2(B_\omega)} \iota_2 \varphi = \tau_{B_\omega} \varphi. \]
As $\varphi$ is a compression of $\varphi_2$ and $\varphi_2$ is nuclear, $\varphi$ is also nuclear.  For each $a \in A_+ \setminus \{0\}$, $\tau_{B_\omega}(\varphi(a)) = \tau_A(a) > 0$ since $\tau_A$ is faithful.  Since $B_\omega$ has strict comparison by Proposition \ref{prop:TraceKernelExtension}, it follows from Lemma 2.2 in \cite{TWW:Annals} that $\varphi$ is full.
\end{proof}

\begin{proposition}\label{prop:UltraproductUniqueness}
Suppose $A$ is a separable, unital, exact $C^*$-algebra satisfying the UCT and $B$ is a simple, unital, $\mathcal{Q}$-stable $C^*$-algebra with a unique trace $\tau_B$ such that every quasitrace on $B$ is a trace and $K_1(B) = 0$.

If $\varphi, \psi : A \rightarrow B_\omega$ are unital, full, nuclear $^*$-homomorphisms such that $K_0(\varphi) = K_0(\psi)$ and $\tau_{B_\omega} \varphi = \tau_{B_\omega} \psi$, then there is a unitary $u \in B_\omega$ such that $\varphi = \mathrm{ad}(u) \psi$.
\end{proposition}

\begin{proof}
Note that $\tau_{B^\omega} q_B \varphi = \tau_{B^\omega} q_B \psi$.  As $B^\omega$ is a $\mathrm{II}_1$-factor by Proposition \ref{prop:TraceKernelExtension} and $q_B \varphi$ and $q_B \psi$ are nuclear, $q_B \varphi$ and $q_B \psi$ are unitarily equivalent by Proposition \ref{prop:FiniteFactorClassification} and a reindexing argument.  Let $\bar{v}$ be a unitary in $B^\omega$ with $q_B \varphi = \mathrm{ad}(\bar{v}) q_B \psi$.  Again as $B^\omega$ is a $\mathrm{II}_1$-factor, the unitary group of $B^\omega$ is path connected.  Hence there is a unitary $v \in B_\omega$ such that $q_B(v) = \bar{v}$.  Replacing $\psi$ with $\mathrm{ad}(v) \psi$, we may assume $q_B \varphi = q_B \psi$.

By Propositions \ref{prop:SeparableCorestriction} and \ref{prop:NonseparableHomGroups}, there is a separable C$^*$-subalgebra $E_0$ of $B_\omega$ containing $\varphi(A)$ and $\psi(A)$ such that the corestrictions of $\varphi$ and $\psi$ to $E_0$ are full and nuclear and agree on $K_0$.  By Proposition \ref{prop:TraceKernelExtension}, $J_B$ is an admissible kernel, $B_\omega$ is separably $\mathcal{Q}$-stable, and $B^\omega$ is a $\mathrm{II}_1$-factor and, therefore, has trivial $K_1$-group.  As having trivial $K_1$-group and being an admissible kernel are separably inheritable properties, Proposition \ref{prop:SeparableExtension} implies there are separable C$^*$-subalgebras $I \subseteq J_B$, $E \subseteq B_\omega$, and $D \subseteq B^\omega$ such that $I$ is an admissible kernel, $E$ contains $E_0$ and is $\mathcal{Q}$-stable, $K_1(D) = 0$, and there is a homomorphism
\[ \begin{tikzcd}
  0 \arrow{r} & I \arrow{r}{\hat{\jmath}} \arrow{d}{\iota_I} & E \arrow{r}{\hat{q}} \arrow{d}{\iota_E} & D \arrow{r} \arrow{d}{\iota_D} & 0 \\
  0 \arrow{r} & J_B \arrow{r}{j_B} & B_\omega \arrow{r}{q_B} & B^\omega \arrow{r} & 0
\end{tikzcd} \]
of extensions where the vertical arrows are the inclusion maps.  Let $\hat{\varphi}, \hat{\psi} : A \rightarrow E$ be the corestrictions of $\varphi$ and $\psi$ to $E$, respectively.  Then $\hat{q} \hat{\varphi} = \hat{q} \hat{\psi}$, $\hat{\varphi}$ and $\hat{\psi}$ are unital, full, and nuclear, and $K_0(\hat{\varphi}) = K_0(\hat{\psi})$.

Let $\lambda : \mathbb{M}_2(E) \rightarrow M(\mathbb{M}_2(I))$ be the canonical $^*$-homomorphism.  Note that the image of $\lambda\iota_2\varphi - \lambda\iota_2\psi$ is contained in $\mathbb{M}_2(I)$, so $(\lambda\iota_2\varphi, \lambda\iota_2\psi)$ is a Cuntz pair in $\mathcal{E}_\mathrm{nuc}(A, \mathbb{M}_2(I))$.  If $\kappa : K_0(A) \rightarrow K_0(\mathbb{M}_2(I))$ is the group homomorphism induced by this Cuntz pair, then $K_0(\hat{\jmath}) \kappa = K_0(\iota_2\hat{\varphi}) - K_0(\iota_2\hat{\psi}) = 0$ by Proposition \ref{prop:KasparovProduct}. As $K_1(D) = 0$, $K_0(\hat{\jmath})$ is injective, and hence $\kappa = 0$.  Proposition \ref{prop:AdmissibleKernel} implies $\mathbb{M}_2(I)$ is an admissible kernel and the canonical group homomorphism
\[ KK_\mathrm{nuc}(A, \mathbb{M}_2(I)) \rightarrow \mathrm{Hom}_{\mathbb{Z}}(K_0(A), K_0(\mathbb{M}_2(I))) \]
is an isomorphism, so $[\lambda\iota_2\hat{\varphi}, \lambda\iota_2\hat{\psi}] = 0$.

As $\hat{\varphi}$ and $\hat{\psi}$ are full and nuclear and $\lambda$ is unital, we have $\lambda \iota_2 \hat{\varphi}$ and $\lambda \iota_2 \hat{\psi}$ are unitizably full and nuclear.  As $\mathbb{M}_2(I)$ is a separable admissible kernel, Proposition \ref{prop:AdmissibleKernel} implies $\mathbb{M}_2(I)$ is stable and satisfies the corona factorization property.  It follows that $\lambda\iota_2\hat{\varphi}$ and $\lambda\iota_2\hat{\psi}$ are nuclearly absorbing by Theorem \ref{thm:ElliottKucerovskyGabe}, and Proposition \ref{prop:KKUniqueness} implies $\iota_2 \hat{\varphi}$ and $\iota_2 \hat{\psi}$ are approximately unitarily equivalent.

Now, $\iota_2 \varphi$ and $\iota_2 \psi$ are approximately unitarily equivalent, and hence, by a reindexing argument, there is a unitary $w \in \mathbb{M}_2(B_\omega)$ with $\iota_2 \varphi = \mathrm{ad}(w) \iota_2 \psi$.  As $\varphi$ and $\psi$ are unital, $w$ commutes with $\iota_2 \varphi(1_A) = \iota_2 \psi(1_A) = 1_{B_\omega} \oplus 0_{B_\omega}$.  Therefore, $w = u \oplus u'$ for unitaries $u$ and $u'$ in $B_\omega$, and $\varphi = \mathrm{ad}(u) \psi$.
\end{proof}

\section{The classification theorem}\label{sec:Classification}

The goal of this section is to show the classification results for $^*$-homomorphisms from $A$ to $B_\omega$ given in Section \ref{sec:UltraproductClassification} imply the analogous results for $^*$-homomorphisms from $A$ to $B$ under the same hypothesis.  The ideas used here go back at least as far as the work of Lin in \cite{Lin:StableApproxEquivalence} and Dadarlat and Eilers in \cite{DadarlatEilers:Classification} and have been used heavily since then.  The uniqueness result for $^*$-homomorphisms from $A$ to $B$ follows immediately from the uniqueness result for $^*$-homomorphisms from $A$ to $B_\omega$.  The difficulty lies in the existence result where only approximately multiplicative maps from $A$ to $B$ can be produced directly from a $^*$-homomorphism from $A$ to $B_\omega$.  To make up for this, a technical uniqueness result for approximately multiplicative maps is needed.

Let $A$ and $B$ be C$^*$-algebras, let $\mathcal{G} \subseteq A$ be a finite set, and let $\delta > 0$ be given. A linear, self-adjoint function $\varphi : A \rightarrow B$ is called \emph{$(\mathcal{G}, \delta)$-multiplicative} if
\[ \| \varphi(aa') - \varphi(a) \varphi(a') \| < \delta \]
for all $a, a' \in \mathcal{G}$.  Following Section 3.3 of \cite{DadarlatEilers:Classification}, a \emph{$K_0$-triple} for a unital C$^*$-algebra $A$ is a triple $(\mathcal{G}, \delta, \mathcal{P})$ where $\mathcal{G} \subseteq A$ is a finite set, $\delta > 0$, and $\mathcal{P} \subseteq \mathcal{P}_\infty(A)$ is a finite set of projections in matrices over $A$ such that whenever $B$ is a unital C$^*$-algebra and $\varphi : A \rightarrow B$ is a linear, self-adjoint, $(\mathcal{G}, \delta)$-multiplicative map, $\| \varphi(p^2) - \varphi(p)^2 \| < 1/4$ for all $p \in \mathcal{P}$.  Note that if $\mathcal{P} \subseteq \mathcal{P}_\infty(A)$ is any finite set, then for all sufficiently large finite sets $\mathcal{G} \subseteq A$ and sufficiently small $\delta > 0$, the triple $(\mathcal{G}, \delta, \mathcal{P})$ is a $K_0$-triple for $A$.

Let $\chi$ denote the characteristic function of $[1/2, \infty)$ defined on the real numbers.  Note that if $(\mathcal{G}, \delta, \mathcal{P})$ is a $K_0$-triple for $A$ and $\varphi : A \rightarrow B$ is a linear, self-adjoint, $(\mathcal{G}, \delta)$-multiplicative map, then $1/2$ is not in the spectrum of $\varphi(p)$ for all $p \in \mathcal{P}$.  Hence for $p \in \mathcal{P}$, we may define $\varphi_\#(p) := [\chi(\varphi(p))] \in K_0(B)$.  In this way, every linear, self-adjoint, $(\mathcal{G}, \delta)$-multiplicative map $\varphi : A \rightarrow B$ defines a function $\varphi_\# : \mathcal{P} \rightarrow K_0(B)$.

\begin{lemma}\label{lemma:ApproximateExistence}
Suppose $A$ is a separable, unital, exact $C^*$-algebra satisfying the UCT and $B$ is a simple, unital, $\mathcal{Q}$-stable $C^*$-algebra with a unique trace $\tau_B$ such that every quasitrace on $B$ is a trace and $K_1(B) = 0$.

If $\tau_A$ is a faithful, amenable trace on $A$ and $\sigma : K_0(A) \rightarrow K_0(B)$ is a group homomorphism such that $\hat{\tau}_B \sigma = \hat{\tau}_A$ and $\sigma([1_A]) = [1_B]$, then for any $K_0$-triple $(\mathcal{G}, \delta, \mathcal{P})$ for $A$, there is a unital, completely positive, nuclear, $(\mathcal{G}, \delta)$-multiplicative map $\varphi : A \rightarrow B$ such that $\varphi_\#(p) = \sigma([p])$ for all $p \in \mathcal{P}$ and $| \tau_B(\varphi(a)) - \tau_A(a) | < \delta$ for all $a \in \mathcal{G}$.
\end{lemma}

\begin{proof}
Suppose $\sigma$ and $\tau_A$ are given as in the statement and let $(\mathcal{G}, \delta, \mathcal{P})$ be a $K_0$-triple for $A$.  Let $\iota_B : B \rightarrow B_\omega$ denote the diagonal embedding.  By Proposition \ref{prop:UltraproductExistence}, there is a unital, nuclear $^*$-homomorphism $\varphi_\omega : A \rightarrow B_\omega$ such that $K_0(\varphi_\omega) = K_0(\iota_B) \sigma$ and $\tau_{B_\omega} \varphi_\omega = \tau_A$.  By the Choi-Effros lifting theorem, there is a sequence of unital, completely positive, nuclear maps $\varphi_n : A \rightarrow B$ representing $\varphi_\omega$.  Let
\[ S_1 = \bigcap_{a, a' \in \mathcal{G}} \{ n \geq 1 : \| \varphi_n(aa') - \varphi_n(a) \varphi_n(a') \| < \delta \} \]
and note that $S_1 \in \omega$ and, for each $n \in S_1$, $\varphi_n$ is $(\mathcal{G}, \delta)$-multiplicative.

Fix $p \in \mathcal{P}$ and let $d, k \geq 1$ be integers such that $p \in \mathbb{M}_d(A)$ and there are projections $e, f \in \mathbb{M}_{k}(B)$ with $\sigma([p]) = [e] - [f]$.  Then $[\varphi_\omega(p)] = [\iota_B(e)] - [\iota_B(f)]$.  Now, there are an integer $\ell \geq 1$ and a partial isometry $v \in \mathbb{M}_{d + k + \ell}(B_\omega)$ such that
\[ v^* v = \varphi_\omega(p) \oplus \iota_B(f) \oplus 1_{B_\omega}^{\oplus \ell} \quad \text{and} \quad v v^* = 0_{B_\omega}^{\oplus d} \oplus \iota_B(e) \oplus 1_{B_\omega}^{\oplus \ell}. \]
Let $(v_n)_{n=1}^\infty$ be a bounded sequence in $\mathbb{M}_{d+k+\ell}(B)$ lifting $v$ and note that
\[ \lim_{n \rightarrow \omega} \| v_n^* v_n - \chi(\varphi_n(p)) \oplus f \oplus 1_B^{\oplus \ell} \| = \lim_{n \rightarrow \omega} \| v_n v_n^* - 0_B^{\oplus d} \oplus e \oplus 1_B^{\oplus \ell} \| = 0. \]
As $p \in \mathcal{P}$ was arbitrary and $\mathcal{P}$ is finite,
\[ S := \bigcap_{p \in \mathcal{P}} \{ n \in S_1 : (\varphi_n)_\#(p) = [\sigma(p)] \} \in \omega. \]

Let $T = \bigcap_{a \in \mathcal{G}} \{ n \geq 1 : |\tau_B (\varphi(a)) - \tau_A(a)| < \delta \}$ and note that $T \in \omega$ since $\tau_{B_\omega} \varphi = \tau_A$.  Now, $S \cap T \in \omega$, and in particular, $S \cap T \neq \emptyset$.  Fix $n \in S \cap T$ and define $\varphi = \varphi_n$.
\end{proof}

\begin{lemma}\label{lemma:ApproximateUniqueness}
Suppose $A$ is a separable, unital, exact $C^*$-algebra satisfying the UCT and $B$ is a simple, unital, $\mathcal{Q}$-stable $C^*$-algebra with a unique trace $\tau_B$ such that every quasitrace on $B$ is a trace and $K_1(B) = 0$.

For any faithful trace $\tau_A$ on $A$, finite set $\mathcal{F} \subseteq A$, and $\varepsilon > 0$, there is a $K_0$-triple $(\mathcal{G}, \delta, \mathcal{P})$ for $A$ such that if $\varphi, \psi : A \rightarrow B$ are unital, completely positive, nuclear, $(\mathcal{G}, \delta)$-multiplicative maps with $\varphi_\#(p) = \psi_\#(p)$, $|\tau_B(\varphi(a)) - \tau_A(a)| < \delta$, and $|\tau_B(\psi(a)) - \tau_A(a)| < \delta$ for all $a \in \mathcal{G}$ and $p \in \mathcal{P}$, then there is a unitary $u \in B$ such that
\[ \| \varphi(a) - u \psi(a) u^* \| < \varepsilon \]
for all $a \in \mathcal{F}$.
\end{lemma}

\begin{proof}
Assume the result is false and fix a faithful trace $\tau_A$ on $A$, a finite set $\mathcal{F} \subseteq A$, and $\varepsilon > 0$ where the result fails.  Let $(\mathcal{G}_n)_{n=1}^\infty$ be an increasing sequence of finite subsets of $A$ with dense union, let $(\delta_n)_{n=1}^\infty$ be a decreasing sequence of positive real numbers converging to zero, and let $(\mathcal{P}_n)_{n=1}^\infty$ be an increasing sequence of finite subsets of $\mathcal{P}_\infty(A)$ with dense\footnote{Here $\mathcal{P}_\infty(A)$ is equipped with the metric induced by the norm on $A \otimes \mathcal{K}$.} union such that $(\mathcal{G}_n, \delta_n, \mathcal{P}_n)$ is a $K_0$-triple for each $n \geq 1$.  For each $n \geq 1$, there are unital, completely positive, nuclear, $(\mathcal{G}_n, \delta_n)$-multiplicative maps $\varphi_n, \psi_n : A \rightarrow B$ such that $(\varphi_n)_\#(p) = (\psi_n)_\#(p)$, $|\tau_B(\varphi_n(a)) - \tau_A(a)| < \delta_n$, and $|\tau_B(\psi_n(a)) - \tau_A(a)| < \delta_n$ for all $a \in \mathcal{G}_n$ and $p \in \mathcal{P}_n$ but such that for each unitary $u_n \in B$, there is an $a \in \mathcal{F}$ with
\[ \| \varphi_n(a) - u_n \psi_n(a) u_n^* \| \geq \varepsilon. \]

Let $\varphi_\omega, \psi_\omega : A \rightarrow B_\omega$ denote the functions induced by the sequences $(\varphi_n)_{n=1}^\infty$ and $(\psi_n)_{n=1}^\infty$, respectively, and note that $\varphi_\omega$ and $\psi_\omega$ are unital $^*$-homomorphisms.  Since $A$ is exact, $\varphi_\omega$ and $\psi_\omega$ are nuclear by Proposition 3.3 in \cite{Dadarlat:QDMorphisms}.  Also, $\tau_{B_\omega} \varphi_{\omega} = \tau_{B_\omega} \psi_\omega = \tau_A$.  Since $\tau_A$ is faithful and $B_\omega$ has strict comparison with respect to $\tau_{B_\omega}$ by Proposition \ref{prop:TraceKernelExtension}, $\varphi_\omega$ and $\psi_\omega$ are full by Lemma 2.2 in \cite{TWW:Annals}.

Fix $p \in \mathcal{P}_n$ and let $d \geq 1$ be an integer with $p \in \mathbb{M}_d(A)$.  As $(\varphi_k)_\#(p) = (\psi_k)_\#(p)$ for all $k \geq n$, $[\chi(\varphi_k(p))] = [\chi(\psi_k(p))]$ in $K_0(B)$ for all $k \geq n$.  Note that $B$ has stable rank one by Corollary 6.6 in \cite{Rordam:UHFStable}, and hence $B$ has cancellation of projections by Proposition 6.5.1 of \cite{Blackadar:KTheory}.  Now, there is a partial isometry $v_k \in \mathbb{M}_d(B)$ with $v_k^*v_k = \chi(\varphi_k(p))$ and $v_k v_k^* = \chi(\psi_k(p))$ for all $k \geq n$.  The sequence $(v_k)_{k=1}^\infty$ defines a partial isometry $v$ in $\mathbb{M}_d(B_\omega)$ with $v^*v = \chi(\varphi_\omega(p)) = \varphi_\omega(p)$ and $vv^* = \chi(\psi_\omega(p)) = \psi_\omega(p)$.  Hence $[\varphi_\omega(p)] = [\psi_\omega(p)]$ in $K_0(B_\omega)$.  This shows $K_0(\varphi_\omega) = K_0(\psi_\omega)$.

Proposition \ref{prop:UltraproductUniqueness} now shows there is a unitary $u \in B_\omega$ with $\varphi_\omega(a) = u \psi_\omega(a) u^*$ for all $a \in A$.  If $(u_n)_{n=1}^\infty \subseteq B$ is a sequence of unitaries lifting $u$, then for each $a \in A$,
\[ \lim_{n \rightarrow \omega} \| \varphi_n(a) - u_n \psi_n(a) u_n^* \| = 0, \]
and in particular, for some integer $n \geq 1$,
\[ \| \varphi_n(a) - u_n \psi_n(a) u_n^* \| < \varepsilon \]
for all $a \in \mathcal{F}$ which is a contradiction.
\end{proof}

Intertwining the previous two lemmas produces the following classification theorem which is the main technical result of the paper.

\begin{theorem}\label{thm:Classification}
Suppose $A$ is a separable, unital, exact $C^*$-algebra satisfying the UCT and $B$ is a simple, unital, $\mathcal{Q}$-stable $C^*$-algebra with a unique trace $\tau_B$ such that every quasitrace on $B$ is a trace and $K_1(B) = 0$.
\begin{enumerate}
  \item If $\tau_A$ is a faithful, amenable trace on $A$ and $\sigma : K_0(A) \rightarrow K_0(B)$ is a group homomorphism such that $\hat{\tau}_B \sigma = \hat{\tau}_A$ and $\sigma([1_A]) = [1_B]$, then there is a unital, faithful, nuclear $^*$-homomorphism $\varphi : A \rightarrow B$ such that $K_0(\varphi) = \sigma$ and $\tau_B \varphi = \tau_A$.
  \item If $\varphi, \psi : A \rightarrow B$ are unital, faithful, nuclear $^*$-homomorphisms with $K_0(\varphi) = K_0(\psi)$ and $\tau_B \varphi = \tau_B \psi$, then there is a sequence of unitaries $(u_n)_{n=1}^\infty \subseteq B$ such that
      \[ \lim_{n \rightarrow \infty} \| \varphi(a) - u_n \psi(a) u_n^* \| = 0 \]
      for all $a \in A$.
\end{enumerate}
\end{theorem}

\begin{proof}
Note that (2) follows immediately from Lemma \ref{lemma:ApproximateUniqueness} with $\tau_A := \tau_B \varphi$.  For (1), fix an increasing sequence of finite sets $\mathcal{F}_n \subseteq A$ with dense union and a sequence $\varepsilon_n > 0$ with $\sum_n \varepsilon_n < \infty$.  Let $(\mathcal{G}_n, \delta_n, \mathcal{P}_n)$ be $K_0$-triples for $A$ satisfying Lemma \ref{lemma:ApproximateUniqueness} for the trace $\tau_A$ and the pair $(\mathcal{F}_n, \varepsilon_n)$.  Enlarging $\mathcal{G}_n$ and $\mathcal{P}_n$ and decreasing $\delta_n$, we may assume the $\mathcal{G}_n$ are increasing with union dense in $A$, the $\delta_n$ are decreasing and converging to zero, and the $\mathcal{P}_n$ are increasing with union dense in $\mathcal{P}_\infty(A)$.

By Lemma \ref{lemma:ApproximateExistence}, for each $n \geq 1$, there is a unital, completely positive, nuclear, $(\mathcal{G}_n, \delta_n)$-multiplicative map $\psi_n : A \rightarrow B$ with $(\psi_n)_\#(p) = \sigma([p])$ for all $p \in \mathcal{P}_n$ and $|\tau_B(\psi_n(a)) - \tau_A(a)| < \delta_n$ for all $a \in \mathcal{G}_n$.  By Lemma \ref{lemma:ApproximateUniqueness}, for each $n \geq 1$, there is a unitary $u_{n+1} \in B$ such that
\[ \| \psi_n(a) - u_{n+1} \psi_{n+1}(a) u_{n+1}^* \| < \varepsilon_n \]
for all $a \in \mathcal{F}_n$.  Define $\varphi_1 = \psi_1$ and define $\varphi_n = \mathrm{ad}(u_2 u_3 \cdots u_n) \psi_n$ for $n \geq 2$.  Then for all $n \geq 1$ and $a \in \mathcal{F}_n$,
\[ \| \varphi_n(a) - \varphi_{n+1}(a) \| = \| \psi_n(a) - u_{n+1} \psi_{n+1}(a) u_{n+1}^* \| < \varepsilon_n. \]
This implies $(\varphi_n(a))_{n=1}^\infty$ is Cauchy for all $a \in \bigcup_{n=1}^\infty \mathcal{F}_n$ whence, by an $\varepsilon/3$ argument, is Cauchy for all $a \in A$.  The desired map $\varphi : A \rightarrow B$ is given by $\varphi(a) = \lim_{n \rightarrow \infty} \varphi_n(a)$.
\end{proof}

The following result provides the rigorous statement of Theorem D.

\begin{corollary}[Theorem D]\label{cor:TheoremD}
Suppose $A$ is a separable, unital, exact $C^*$-algebra satisfying the UCT and $B$ is a simple, unital AF-algebra with a unique trace $\tau_B$ and divisible $K_0$-group.
\begin{enumerate}
  \item If $\tau_A$ is a faithful, amenable trace on $A$ and $\sigma : K_0(A) \rightarrow K_0(B)$ is a group homomorphism such that $\hat{\tau}_B \sigma = \hat{\tau}_A$ and $\sigma([1_A]) = [1_B]$, then there is a unital, faithful $^*$-homomorphism $\varphi : A \rightarrow B$ such that $K_0(\varphi) = \sigma$ and $\tau_B \varphi = \tau_A$.
  \item If $\varphi, \psi : A \rightarrow B$ are unital, faithful $^*$-homomorphisms with $K_0(\varphi) = K_0(\psi)$ and $\tau_B \varphi = \tau_B \psi$, then there is a sequence of unitaries $(u_n)_{n=1}^\infty \subseteq B$ such that
      \[ \lim_{n \rightarrow \infty} \| \varphi(a) - u_n \psi(a) u_n^* \| = 0 \]
      for all $a \in A$.
\end{enumerate}
\end{corollary}

\begin{proof}
As $B$ is an AF-algebra, we have that $K_0(B)$ is torsion-free, $K_1(B) = 0$, and every quasitrace on $B$ is a trace.  As $K_0(B)$ is divisible and torsion-free, the homomorphism $K_0(B) \rightarrow K_0(B) \otimes_{\mathbb{Z}} \mathbb{Q}$ given by $[p] \mapsto [p] \otimes 1_\mathbb{Q}$ for all projections $p \in B \otimes \mathcal{K}$ is an isomorphism.  Also, viewing $B \otimes \mathcal{Q}$ as the inductive limit of the diagonal embeddings $\mathbb{M}_k(B) \rightarrow \mathbb{M}_\ell(B)$ for integers $k, \ell \geq 1$ with $k$ dividing $\ell$, the group homomorphisms
\[ K_0(\mathbb{M}_k(B)) \rightarrow K_0(B) \otimes_\mathbb{Z} \mathbb{Q} : [p] \mapsto [p] \otimes \frac1k \]
induce an isomorphism $K_0(B \otimes \mathcal{Q}) \rightarrow K_0(B) \otimes_\mathbb{Z} \mathbb{Q}$.  These isomorphisms induce an isomorphism $K_0(B) \cong K_0(B \otimes \mathcal{Q})$ of ordered abelian groups, and it follows from Elliott's classification of AF-algebras that $B \cong B \otimes \mathcal{Q}$.  As $B$ is nuclear, all $^*$-homomorphisms $A \rightarrow B$ are nuclear, so the result follows from Theorem \ref{thm:Classification}.
\end{proof}

\section{Applications}\label{sec:Applications}

\begin{proof}[Proof of Theorem A]
We may assume $A$ is unital.  Let $\tau_A$ be a faithful, amenable trace on $A$ and let
\[ G_0 = \mathrm{span}_\mathbb{Q} \{ (\tau_A \otimes \mathrm{Tr}_\mathcal{K}) (p) : p \in A \otimes \mathcal{K} \text{ is a projection} \} \subseteq \mathbb{R}. \]
By \cite{Elliott:TotallyOrderedK0}, if $G_0$ is equipped with the order and unit inherited from $\mathbb{R}$, then $G_0$ is a simple dimension group with a unique state given by the inclusion $G_0 \hookrightarrow \mathbb{R}$.  Let $B$ be a unital AF-algebra with $K_0(B) \cong G_0$ and note that $B$ is simple and has a unique trace.  Composing the map $K_0(A) \rightarrow G_0$ induced by the trace with an isomorphism $G_0 \rightarrow K_0(B)$ produces a group homomorphism $\sigma : K_0(A) \rightarrow K_0(B)$ compatible with the unit and the trace.  Corollary \ref{cor:TheoremD} now implies the existence of a unital, trace-preserving embedding $A \rightarrow B$.
\end{proof}

\begin{proof}[Proof of Theorem B]
The implications (1) $\Leftrightarrow$ (2) $\Leftarrow$ (3) are well known.  To show (1) implies (3), let $G$ be a countable, discrete, amenable group.  A result of Higson and Kasparov in \cite{HigsonKasparov:BaumConnes} shows $G$ satisfies the Baum-Connes conjecture, and hence a result of L\"uck in \cite{Luck:TraceConjecture} shows that if $\tau_{\mathrm{C}^*_r(G)}$ denote the canonical trace on $\mathrm{C}^*_r(G)$, then for all projections $p \in \mathrm{C}^*_r(G) \otimes \mathcal{K}$, $(\tau_{\mathrm{C}^*_r(G)} \otimes \mathrm{Tr}_\mathcal{K})(p) \in \mathbb{Q}$.  As $K_0(\mathcal{Q}) \cong \mathbb{Q}$, the trace on $\mathrm{C}^*_r(G)$ produces a group homomorphism $K_0(\mathrm{C}^*_r(G)) \rightarrow K_0(\mathcal{Q})$ compatible with the unit and trace.  A result of Tu in \cite{Tu:AmenableGroupoid} shows $\mathrm{C}^*_r(G)$ satisfies the UCT, and the result follows from Corollary \ref{cor:TheoremD}.
\end{proof}

\begin{proof}[Proof of Theorem C]
If $C_0(X) \rtimes_r G$ embeds into a simple, unital AF-algebra $B$, then any trace on $B$ induces a faithful, $G$-invariant, Borel measure on $X$ with mass at most 1, and hence after rescaling, $X$ admits a probability measure of the desired form.

Conversely, suppose $G$ is amenable and suppose $\mu$ is a faithful, $G$-invariant, Borel, probability measure on $X$.  If $E : C_0(X) \rtimes G \rightarrow C_0(X)$ is the canonical conditional expectation, then the map $C_0(X) \rtimes G \rightarrow \mathbb{C}$ given by
\[ a \mapsto \int_X E(a) \, d\mu \]
is a faithful trace on $C_0(X) \rtimes G$ which is amenable as $C_0(X) \rtimes_r G$ is nuclear.  A result of Tu in \cite{Tu:AmenableGroupoid} shows $C_0(X) \rtimes_r G$ satisfies the UCT, so the result follows from Theorem A.
\end{proof}

In \cite{Lin:AFEmbAHalg}, Lin has shown that if $A$ is an AH-algebra and $\alpha$ is an action of $\mathbb{Z}^k$ on $A$ such that $A$ admits a faithful, $\alpha$-invariant trace, then the crossed product $A \rtimes_r \mathbb{Z}^k$ embeds into an AF-algebra.  Theorem A implies the following generalization of this result.

\begin{corollary}
Suppose $A$ is a separable, exact $C^*$-algebra satisfying the UCT, $G$ is a countable, discrete, torsion-free, abelian group, and $\alpha$ is an action of $G$ on $A$.  If $A$ admits a faithful, $\alpha$-invariant, amenable trace, then $A \rtimes_\alpha G$ embeds into an AF-algebra.
\end{corollary}

\begin{proof}
Let $G_n$ be an increasing sequence of finitely generated subgroups of $G$ with union $G$.  If $\alpha_n$ is the restriction of $\alpha$ to an action of $G_n$ on $A$, then the canonical $^*$-homomorphism
\[ \underset{\longrightarrow}{\lim} \, A \rtimes_{\alpha_n} G_n \longrightarrow A \rtimes_\alpha G \]
is an isomorphism.  As each $G_n$ is finitely generated, torsion-free, and abelian, $G_n \cong \mathbb{Z}^{d(n)}$ for some integer $d(n) \geq 1$.  It follows that $A \rtimes_\alpha G$ satisfies the UCT.  Also, $A$ is separable and exact, and if $E : A \rtimes_\alpha G \rightarrow A$ is the canonical conditional expectation and $\tau_A$ is a faithful, $\alpha$-invariant, amenable trace on $A$, then $\tau_A E$ is a faithful trace on $A \rtimes_\alpha G$.  As $\tau_A$ is amenable and $A$ is exact, $\pi_{\tau_A}(A)''$ is injective.  Therefore,
\[ \pi_{\tau_A E}(A \rtimes_\alpha G)'' \cong \pi_{\tau_A}(A)'' \bar{\rtimes} G, \]
is injective by Proposition 6.8 of \cite{Connes:InjectiveFactors}, and hence $\tau_A E$ is amenable.  The result now follows from Theorem A.
\end{proof}

The AF-embedding problem for C$^*$-algebras of countable 1-graphs was solved in \cite{Schafhauser:Graph1}.  For countable, cofinal, row-finite 2-graphs with no sources, a similar result was obtained by Clark, an Huef, and Sims in \cite{ClarkHuefSims}.  Using their techniques together with Theorem A, the main result of \cite{ClarkHuefSims} extends to $k$-graphs.

\begin{corollary}
Let $k \geq 1$ be an integer, let $\Lambda$ be a countable, cofinal, row-finite $k$-graph with no sources, and let $A_1, \ldots, A_k$ denote the coordinate matrices of $\Lambda$.  The following are equivalent:
\begin{enumerate}
  \item $\mathrm{C}^*(\Lambda)$ embeds into an AF-algebra;
  \item $\mathrm{C}^*(\Lambda)$ is quasidiagonal;
  \item $\mathrm{C}^*(\Lambda)$ is stably finite;
  \item $\left(\sum_{i=1}^k \mathrm{im}(1 - A_i^t) \right) \cap \mathbb{Z}_+ \Lambda^0 = \{0\}$;
  \item $\Lambda$ admits a faithful graph trace.
\end{enumerate}
\end{corollary}

\begin{proof}
The equivalence of (2) through (5) is Theorem 1.1 in \cite{ClarkHuefSims}, and it is well known that (1) implies (2).  The same proof given in Lemma 3.7 of \cite{ClarkHuefSims} shows (5) implies (1) by appealing to Theorem A above in place of Corollary B of \cite{TWW:Annals}.
\end{proof}

A C$^*$-algebra $A$ is \emph{matricial field} (\emph{MF}) if there is a net $\varphi_i : A \rightarrow \mathbb{M}_{n(i)}$ of linear, self-adjoint functions such that
\[ \| \varphi_i(aa') - \varphi_i(a) \varphi_i(a') \| \rightarrow 0 \quad \text{and} \quad \| \varphi_i(a) \| \rightarrow \|a\| \]
for all $a, a' \in A$.  Similarly, a trace $\tau_A$ on a C$^*$-algebra $A$ is called \emph{matricial field} (\emph{MF}) if there is a net $\varphi_i : A \rightarrow \mathbb{M}_{n(i)}$ of linear, self-adjoint functions such that
\[ \| \varphi_i(aa') - \varphi_i(a) \varphi_i(a') \| \rightarrow 0 \quad \text{and} \quad \tau_{\mathbb{M}_{n(i)}}(\varphi_i(a)) \rightarrow \tau_A(a) \]
for all $a, a' \in A$.

The class of MF algebras was introduced by Blackadar and Kirchberg in \cite{BlackadarKirchberg:MFAlgebras}.  Haagerup and Thorbj{\o}rnsen have shown in \cite{HaagerupThorbjornsen} that $\mathrm{C}^*_r(F_2)$ is MF where $F_2$ is the free group on two generators.  From here, it follows as well that $\mathrm{C}^*_r(F)$ is MF for all free groups $F$ (one first reduces to the case where $F$ is countable and then considers an embedding $F \hookrightarrow F_2$).  Using this result, many reduced crossed products of C$^*$-algebras by free groups have been shown to be MF in \cite{KerrNowak, Rainone:MF, RainoneSchafhauser, Schafhauser:MF}.  The results above allow for substantial generalizations of these results.

\begin{corollary}\label{cor:MFTrace}
Suppose $A$ is a separable, exact $C^*$-algebra satisfying the UCT and $\alpha$ is an action of a free group $F$ on $A$.  If $\tau$ is a trace on $A \rtimes_r F$ such that $\tau|_A$ is faithful and amenable, then $\tau$ is MF.
\end{corollary}

\begin{proof}
Adding a unit to $A$ if necessary, we may assume $A$ is unital.  By Theorem 4.8 in \cite{RainoneSchafhauser}, there is a group homomorphism $\sigma : K_0(A) \rightarrow K_0(\mathcal{Q}_\omega)$ such that $\sigma([1_A]) = [1_{\mathcal{Q}_\omega}]$, $\hat{\tau}_{\mathcal{Q}_\omega} \sigma = \hat{\tau}|_A$, and $\sigma K_0(\alpha_s) = \sigma$ for each $s \in F$.  By Proposition \ref{prop:UltraproductExistence}, there is a unital, full, nuclear $^*$-homomorphism $\varphi : A \rightarrow \mathcal{Q}_\omega$ such that $K_0(\varphi) = \sigma$ and $\tau_{\mathcal{Q}_\omega} \varphi = \tau|_A$.

For each $s$ in a free generating set for $F$, we have $K_0(\varphi \alpha_s) = K_0(\varphi)$ and $\tau_{\mathcal{Q}_\omega} \varphi \alpha_s = \tau_{\mathcal{Q}_\omega} \varphi$, and hence, by Proposition \ref{prop:UltraproductUniqueness}, there is a unitary $u_s \in \mathcal{Q}_\omega$ such that $\varphi \alpha_s = \mathrm{ad}(u_s) \varphi$.  As $F$ is free, the function $s \mapsto u_s$ extends to a unitary representation of $F$ on $\mathcal{Q}_\omega$ with $\varphi \alpha_s = \mathrm{ad}(u_s) \varphi$ for all $s \in F$.

This shows the trace $\tau|_A$ is $\alpha$-MF in the sense of \cite{RainoneSchafhauser}, and hence, if $E : A \rtimes_r F \rightarrow A$ denotes the conditional expectation, then the trace $(\tau|_A) E$ is MF.  When $F \not\cong \mathbb{Z}$, $\tau = (\tau|_A) E$ by \cite{HarpeSkandalis}, and therefore, $\tau$ is MF.

Now assume $F = \mathbb{Z}$.  Since $A$ is exact and $\tau|_A$ is amenable, $\pi_{\tau|_A}(A)''$ is injective.  Since
\[ \pi_{(\tau|_A) E}(A \rtimes \mathbb{Z})'' \cong \pi_{\tau|_A}(A)'' \bar{\rtimes} \mathbb{Z} \]
is injective by Proposition 6.8 of \cite{Connes:InjectiveFactors}, $(\tau|_A) E$ is amenable.  Now, $A \rtimes \mathbb{Z}$ is separable and exact, satisfies the UCT, and admits a faithful, amenable trace $(\tau|_A) E$.  By Theorem 3.8 in \cite{Gabe:TWW}, $(\tau|_A) E$ is quasidiagonal, and hence $A \rtimes \mathbb{Z}$ is quasidiagonal.  By Theorem 4.1 in \cite{Gabe:TWW}, every amenable trace on $A \rtimes \mathbb{Z}$ is quasidiagonal and, in particular, is MF.  So it suffices to show $\tau$ is amenable.

Let $T(A)$ denote the set of traces on $A$.  Since $A$ is unital, $T(A)$ is a Choquet simplex by Theorem 3.1.18 in \cite{Sakai}.  It was shown by Kirchberg in Lemma 3.4 of \cite{Kirchberg:FactorizationProperty} that the set of amenable traces on $A$ is a weak$^*$-closed face in $T(A)$, so by Corollary II.5.20 of \cite{Alfsen}, there is a continuous, affine function $f : T(A) \rightarrow [0, 1]$ such that for $\rho \in T(A)$, $f(\rho) = 0$ if, and only if, $\rho$ is amenable.  Let $\hat{\alpha}$ be the action of $\mathbb{T}$ on $A \rtimes \mathbb{Z}$ dual to $\alpha$.  Then
\[ (\tau|_A) E = \int_\mathbb{T} \tau \hat{\alpha}_z \, d z. \]
As $f((\tau|_A) E) = 0$ and the map $z \mapsto f(\tau \hat{\alpha}_z)$ is positive and continuous on $\mathbb{T}$, we have $f(\tau \hat{\alpha}_z) = 0$ for all $z \in \mathbb{T}$ by the faithfulness of the Lebesgue measure on $\mathbb{T}$.  Therefore, $\tau \hat{\alpha}_z$ is an amenable trace for all $z \in \mathbb{T}$, and in particular, $\tau = \tau \hat{\alpha}_1$ is an amenable trace.
\end{proof}

\begin{corollary}\label{cor:TypeIMFTrace}
If $A$ is a locally type I $C^*$-algebra and $\alpha$ is an action of a free group $F$ on $A$, then every trace on $A \rtimes_r F$ is MF.
\end{corollary}

\begin{proof}
We may assume $A$ is separable.  Suppose $\tau$ is a trace on $A \rtimes_r F$.  Then $J := \{ a \in A : \tau(a^*a) \} = 0$ is an $\alpha$-invariant ideal of $A$ and $\tau$ vanishes on the ideal $J \rtimes_r F$ of $A \rtimes_r F$.  Hence, as $F$ is exact, $\tau$ factors through $(A / J) \rtimes_r F$.  As $A$ is locally type I, so is $A / J$.  After replacing $A$ with $A / J$, we may assume $\tau|_A$ is faithful.  By Theorem 1.1 in \cite{Dadarlat:locallyUCT}, locally type I algebras satisfy the UCT.  Since locally type I algebras are nuclear, the result follows from Corollary \ref{cor:MFTrace}.
\end{proof}

The following is a substantial generalization of a result of Kerr and Nowak (Theorem 5.2 in \cite{KerrNowak}) which gives the same statement in the case when $A$ is commutative.  An action of a group $G$ on a C$^*$-algebra $A$ is called \emph{minimal} if $A$ admits no non-trivial $G$-invariant ideals.

\begin{corollary}\label{cor:KerrNowak}
Suppose $A$ is a separable, unital, nuclear $C^*$-algebra satisfying the UCT and $F$ is a free group acting minimally on $A$.  Then the following are equivalent:
\begin{enumerate}
  \item $A \rtimes_r F$ is MF;
  \item $A \rtimes_r F$ is stably finite;
  \item $A$ admits an invariant trace.
\end{enumerate}
\end{corollary}

\begin{proof}
All MF C$^*$-algebras are stably finite, so (1) implies (2).  For (2) implies (3), if $A \rtimes_r F$ is stably finite, then since $A \rtimes_r F$ is unital and exact, $A \rtimes_r F$ admits a trace $\tau$ by Corollary 5.12 in \cite{Haagerup:Quasitrace}, and then $\tau|_A$ is an invariant trace on $A$.  Finally, for (3) implies (1), suppose $A$ admits an invariant trace $\tau_A$.  Since $\{ a \in A : \tau_A(a^*a) = 0 \}$ is an invariant ideal in $A$ and $\alpha$ is minimal, $\tau_A$ is faithful.  As $A$ is nuclear, $\tau_A$ is amenable, and hence $\tau_A E$ is MF by Corollary \ref{cor:MFTrace}.  Since $\tau_A$ and $E$ are faithful, $\tau_A E$ is faithful.  So $A \rtimes_r F$ has a faithful, MF trace whence is MF.
\end{proof}

Theorem D also allows for a description of closed unitary orbits of normal operators in simple, unital AF-algebras with a unique trace and divisible $K_0$-group in terms of spectral data.

\begin{corollary}
Suppose $B$ is a simple, unital AF-algebra with a unique trace and divisible $K_0$-group.  Two normal operators $a$ and $b$ in $B$ are approximately unitarily equivalent if, and only if, $\sigma(a) = \sigma(b)$, $\tau_B(f(a)) = \tau_B(f(b))$ for all $f \in C(\sigma(a))$, and for every compact, open set $U \subseteq \sigma(a)$, the projections $\chi_U(a)$ and $\chi_U(b)$ are unitarily equivalent.
\end{corollary}

\begin{proof}
Let $X = \sigma(a) = \sigma(b)$ and let $\hat{a}$ and $\hat{b}$ be the $^*$-homomorphisms $C(X) \rightarrow B$ given by the functional calculus.  If $X = \sigma(a) = \sigma(b)$, then $K_0(C(X))$ is canonically isomorphic to $C(X, \mathbb{Z})$ and is generated as a group by the elements $[\chi_U]$ for compact, open sets $U \subseteq X$.  Hence $K_0(\hat{a}) = K_0(\hat{b})$ if, and only if, $\chi_U(a)$ and $\chi_U(b)$ are unitarily equivalent for each compact, open set $U \subseteq X$.  Now apply the uniqueness portion of Theorem \ref{thm:Classification} to the $^*$-homomorphisms $\hat{a}$ and $\hat{b}$.
\end{proof}

We end with an abstract characterization of AF-algebras among the class of simple, unital C$^*$-algebras with a unique trace and divisible $K_0$-group.  The following result is known and can be deduced from the classification of separable, simple, unital, nuclear, $\mathcal{Z}$-stable C$^*$-algebras which satisfy the UCT and have a unique trace.  The minimal route through the literature seems to be the quasidiagonality theorem of \cite{TWW:Annals}, the results of Matui-Sato to show $A$ is tracially AF \cite{MatuiSato:DecompositionRank}, and the classification of separable, simple, unital, nuclear, tracially AF-algebras satisfying the UCT due to Lin in \cite{Lin:TAFClassification}.  The latter result makes heavy use of deep classification and structural results for approximately homogeneous C$^*$-algebras.  It is worth emphasizing that the proof given here does not depend on tracial approximations of any kind and does not depend on any inductive limit structure beyond that for AF-algebras.

\begin{corollary}\label{cor:AFAlg}
Suppose $A$ is a simple, unital, $C^*$-algebra with a unique trace and divisible $K_0$-group.  Then $A$ is an AF-algebra if, and only if, $A$ is separable, nuclear, and $\mathcal{Q}$-stable, $A$ satisfies the UCT, and $K_1(A) = 0$.
\end{corollary}

\begin{proof}
It is well known that AF-algebras are separable, nuclear, and satisfy the UCT.  As $K_0(A)$ is divisible, $K_0(A \otimes \mathcal{Q}) \cong K_0(A)$ and hence $A$ is $\mathcal{Q}$-stable by the classification of AF-algebras.

Conversely, suppose $A$ is a separable, simple, unital, nuclear, $\mathcal{Q}$-stable C$^*$-algebra with a unique trace, $A$ satisfies the UCT, and $K_1(A) = 0$.  As $A$ is simple, the trace on $A$ is necessarily faithful, so $A$ is stably finite.  Hence $(K_0(A), K_0^+(A), [1_A])$ is an ordered group by Proposition 6.3.3 in \cite{Blackadar:KTheory}.

As $A$ is $\mathcal{Q}$-stable and has a unique trace, $A$ has real rank zero by Theorem 7.2 of \cite{Rordam:UHFStable2}.  Also, $A$ has stable rank one by Corollary 6.6 in \cite{Rordam:UHFStable} and, therefore, has cancellation of projections by Proposition 6.5.1 of \cite{Blackadar:KTheory}.  Therefore, $K_0(A)$ has Riesz interpolation by Corollary 1.6 in \cite{Zhang:RieszInterpolation}.  Also, as $A$ is $\mathcal{Q}$-stable, $K_0(A)$ is unperforated.  By the Effros-Handelman-Shen Theorem \cite{EffrosHandelmanShen}, there is a unital AF-algebra $B$ with $K_0(A) \cong K_0(B)$ as ordered groups.  Necessarily, $B$ is simple and has a unique trace.  Let $\sigma : K_0(A) \rightarrow K_0(B)$ be a unital order isomorphism.  As $K_0(B)$ has unique state, we have $\hat{\tau}_B \sigma = \hat{\tau}_A$.  Applying Theorem \ref{thm:Classification}, there is a unital, nuclear $^*$-homomorphism $\varphi : A \rightarrow B$ with $K_0(\varphi) = \sigma$ and $\tau_B \varphi = \tau_A$.

As $B$ is an AF-algebra and $A$ has stable rank one, there is a $^*$-homomorphism $\psi : B \rightarrow A$ such that $K_0(\psi) = \sigma^{-1}$ and $\tau_A \psi = \tau_B$ by the classification of $^*$-homomorphisms out of AF-algebras (or by Theorem \ref{thm:Classification}).  Using Theorem \ref{thm:Classification}, $\psi \varphi$ is approximately unitarily equivalent to the identity on $A$ and $\varphi \psi$ is approximately unitarily equivalent to the identity on $B$.  By Elliott's intertwining argument (see Corollary 2.3.4 of \cite{Rordam:YellowBook}), $A \cong B$.
\end{proof}

\end{document}